\documentclass[11pt]{article}
\usepackage[top=1in, bottom=1in, left=1in, right=1in]{geometry}
\usepackage{paper}
\usepackage{multirow}
\usepackage{nicefrac}

\usepackage{todonotes} 
\usepackage{comment}

\renewcommand{\vec}[3]{{#1}^{(#2)}_{#3}}
\newcommand{\func}[2]{#1\bigl(#2\bigr)}
\newcommand{\codefun}[1]{\text{#1}}

\newcommand{\strongcvx}[1]{strongly $\norm{\cdot}_#1$-convex}
\newcommand{\strongcvxity}[1]{strongly $\norm{\cdot}_#1$-convexity}

\begin{document}

\title{Coordinate-wise descent methods for leading eigenvalue problem}

\author{Yingzhou Li$^{\,\sharp}$,
        Jianfeng Lu$^{\,\sharp\,\dagger}$,
        Zhe Wang$^{\,\sharp}$
  \vspace{0.1in}\\
  $\sharp$ Department of Mathematics, Duke University\\
  $\dagger$ Department of Chemistry and Department of Physics,
  Duke University\\
}

\maketitle

\begin{abstract} 
  Leading eigenvalue problems for large scale matrices arise in many
  applications. Coordinate-wise descent methods are considered in
  this work for such problems based on a reformulation of the leading
  eigenvalue problem as a non-convex optimization
  problem. The convergence of several coordinate-wise methods is analyzed
  and compared. Numerical examples of applications to quantum many-body
  problems demonstrate the efficiency and provide benchmarks of the
  proposed coordinate-wise descent methods.
\end{abstract}


\tableofcontents

\reversemarginpar
\section{Introduction}
\label{sec:Intro}

This paper focuses on solving the leading eigenvalue problem (LEVP)
via coordinate-wise descent methods. Given a symmetric matrix $A$,
the LEVP is defined as
\begin{equation} \label{eq:LEVP}
    A x = \lambda x,
\end{equation}
for $A \in \bbR^{n \times n}$, $A^\top = A$, $\lambda$ is the largest
eigenvalue of $A$, and $x \in \bbR^n$ is the corresponding eigenvector
(we assume here and in the sequel that the leading eigenvalue is
positive and non-degenerate). In the case that several leading
eigenpairs are needed, we can combine the methods for LEVP together
with deflation technique to obtain all desired eigenpairs.
If the eigenvalue of largest
magnitude is needed, we can apply the proposed methods to both $A$ and
$-A$ and then select the desired one.

Leading eigenvalue(s) problems appear in a wide range of applications,
including principal component analysis (PCA), spectral clustering,
dimension reduction, electronic structure calculation, quantum many-body
problems, etc. As a result, many methods have been developed to
address the leading eigenvalue(s) problems, e.g., power method, Lanczos
algorithm~\cite{Golub2013, Lanczos1950}, randomized SVD~\cite{Halko2011},
(Jacobi-)Davidson algorithm~\cite{Davidson1975,
  Sleijpen1996}, local optimal block preconditioned conjugate gradient
(LOBPCG)~\cite{Knyazev2001a}, projected preconditioned conjugate
gradient (PPCG)~\cite{Vecharynski2015}, orbital minimization
method (OMM)~\cite{Corsetti2014, Lu2017a}, Gauss-Newton
algorithm~\cite{Liu2015c}, etc.  However, most of those traditional
iterative methods apply the matrix $A$ every iteration and require
many iterations before convergence, where the iteration number usually
depends on the condition number of $A$ or the leading eigengap of $A$.
Some other methods conduct the inverse of the matrix every iterations,
which is too expensive for large $A$.  In this paper, we focus on
designing algorithms for solving LEVP of extremely large size and
expensive matrix evaluation.  Full configuration interaction (FCI) from
quantum many-body problems is one of the examples in such a setting. In
FCI calculation \cite[pp.~350--353]{SzaboOstlund1996}, the size of
the Hamiltonian matrix grows factorially as the number of electrons
and orbitals in the system increases. For such a large dimensional
problem, it is infeasible to store the matrix in the main memory and
matrix entry should be computed on-the-fly. Thus the evaluation of
the matrix entry is expensive.  Another property of FCI is that the
magnitudes of coordinates of the leading eigenvector span widely. Many
coordinates are relatively negligible. Hence traditional methods, which
treat all coordinates evenly, are unfavorable for FCI problems. Novel
algorithms adapted to the properties of FCI are desired. Although the
designed algorithms in this paper can be applied to other problems,
FCI will be regarded as the target application.

The LEVP naturally can be rewritten as the following unconstrained
optimization problem,
\begin{equation} \label{eq:LEVP-opt}
    \min_{x \in \bbR^n} \norm{A - x x^\top}_F^2,
\end{equation}
where $\norm{\cdot}_F$ denotes the Frobenius norm. Throughout this paper,
we denote $f(x) = \norm{A - x x^\top}_F^2$ as the objective function.
Based on the eigendecomposition of $A$, it is easy to verify that $x
= \pm \sqrt{\lambda_1} v_1$ are minimizers of \eqref{eq:LEVP-opt},
where $\lambda_1$ is the largest eigenvalue of $A$ and $v_1$ is
the corresponding eigenvector with $\norm{v_1} = 1$. Therefore,
if the optimization problem \eqref{eq:LEVP-opt} can be solved,
the leading eigenpair $(\lambda_1, v_1)$ can be reconstructed from
the minimizer $x^*$, i.e., $\lambda_1 = \norm{x^*}^2$ and $v_1 =
\frac{x^*}{\norm{x^*}}$.  Such an unconstrained optimization problem
\eqref{eq:LEVP-opt} has appeared before for solving eigenvalue
problems~\cite{Lei2016, Liu2015c}. In this paper, we solve the LEVP
problem \eqref{eq:LEVP-opt} with algorithms adapted to FCI problems.

Recently, as the rise of big data and deep learning, stochastic iterative
methods are revived and revisited to reduce the computational cost per
iteration, allow aggressive choice of stepsize, and fully adopt the
modern computer architecture. Stochastic gradient descent methods and
coordinate-wise descent methods (CDMs) are two popular groups of methods
among them. Considering CDMs, when single coordinate is updated every
iteration, they only access one column of $A$ each iteration and hence
cost $\bigO(n)$ operations per-iteration for $n$ being the size of the
problem. Meanwhile, the stepsize could be often $n$ times bigger than
that in the traditional gradient descent method for convex objective
function~\cite{Nesterov2012, Richtarik2014, Wright2015}.  Massive
parallelization capability is also observed~\cite{Liu2015a, Liu2015b,
Parnell2018, Tappenden2018, You2016}, which is achieved mostly via its
``asynchronized'' property of different coordinates.~\footnote{Although
different coordinates can be updated in an asynchronized way, these
stochastic iterative methods still requires synchronization every now
and then.} In this work, we develop and analyze methods within the
scope of coordinate-wise descent methods addressing the LEVP.

\subsection{Coordinate-wise descent method}
\label{sec:introCDM}

The history of coordinate-wise methods (CDMs) dates back to the early
development of the discipline of optimization~\cite{Desopo1959}.  It did
not catch much attention until very recently.  Readers are referred to
the recent surveys \cite{Shi2016, Wright2015} and references therein for
detailed history of the development and analysis of general CDMs. Since
the survey papers, the area of CDMs is still under fast development, see
e.g., \cite{Johnson2017, Perekrestenko2017, Wright2017, Xu2018}. Momentum
acceleration can also be combined with the CDMs to further accelerate the
convergence~\cite{Allen-Zhu2016,Lee2013,Lin2015,Nesterov2012}. Besides
these new development of methodology, new applications adopt CMDs to
accelerate computations, including but not limit to the area of imaging
processing~\cite{Han2018}, signal processing~\cite{AlaeeKerahroodi2017,
Mitchell2016, Zeng2017}, wireless communication~\cite{Wu2016, Wu2016a},
data science~\cite{Shi2017, Tu2016, Vandaele2016}, etc. It is worth
mentioning that \citet{Peng2016} discussed coordinate friendly structure,
which could be beneficial for more applications.

In terms of designing CDMs for LEVP, \citet{Lei2016} proposes
coordinate-wise power method (CPM) addressing the LEVP. CPM accelerates
the traditional power method by the coordinate-wise updating
technique. In the same paper, a symmetric greedy coordinate descent
(SGCD) method is proposed, which will be reviewed in detail in the
following. \citet{Wang2017a} adopts shift-and-invert power method to
solve the LEVP, where the inverse of the shifted linear system is
addressed by coordinate-wise descent method. Different coordinate
updating rules are employed, i.e., Gauss-Southwell-Lipschitz rule
(SI-GSL), cyclic rule (SI-Cyclic), and accelerated randomized coordinate
descent methods (SI-ACDM).  All these methods are proposed as methods
calculating the leading eigenvector and outperforms many other method
when an accurate upper bound of the leading eigenvalue is given. It is
not clear how to obtain such upper bound efficiently in practice however.

\begin{algorithm}[htp]
    \caption{General coordinate-wise descent method for LEVP}
    \label{alg:CDM}

    Input: Symmetric matrix $A \in \bbR^{n \times n}$; initial vector
    $\vec{x}{0}{}$.

\begin{algorithmic}[1]

    \STATE $\vec{z}{0}{} = A\vec{x}{0}{}$

    \STATE $\ell = 0$
    \WHILE {(not converged)}
        \STATE $j_\ell = \codefun{coordinate-pick}( \vec{x}{\ell}{},
        \vec{z}{\ell}{} )$
        \STATE $\vec{x}{\ell+1}{j} = \left\{
            \begin{array}{ll}
                \codefun{coordinate-update}( \vec{x}{\ell}{},
                \vec{z}{\ell}{}, j_\ell ), & j=j_\ell \\
                \vec{x}{\ell}{j},
                & j \neq j_\ell
            \end{array}
            \right.$
        \STATE $\vec{z}{\ell+1}{} = \vec{z}{\ell}{} + A_{:,j_\ell}\left(
        \vec{x}{\ell+1}{j_\ell} - \vec{x}{\ell}{j_\ell} \right)$

        \STATE $\ell = \ell+1$
    \ENDWHILE
\end{algorithmic}

\end{algorithm}

A general coordinate-wise descent method (CDM) for the
LEVP can be summarised as Algorithm~\ref{alg:CDM}. Given a
symmetric matrix $A$ and the initial vector $\vec{x}{0}{}$. CDM
first picks a coordinate $j_\ell$ according to a specific rule
``$\codefun{coordinate-pick}(\vec{x}{\ell}{}, \vec{z}{\ell}{})$''.
Then it updates the $j_\ell$-th coordinate of $\vec{x}{\ell}{}$ with
``$\codefun{coordinate-update}(\vec{x}{\ell}{}, \vec{z}{\ell}{},
j_\ell)$'' and reaches the vector $\vec{x}{\ell+1}{}$ for the next
iteration. One key to obtain a good choice of $j_\ell$ and the
update is to get $\vec{z}{\ell}{} = A \vec{x}{\ell}{}$ involved in
the calculation. Since $\vec{x}{\ell+1}{}$ and $\vec{x}{\ell}{}$ only
differ by a single coordinate, we have an efficient updating expression
for $\vec{z}{\ell}{}$,
\begin{equation*}
    \vec{z}{\ell+1}{} = A\vec{x}{\ell+1}{} = A \left( \vec{x}{\ell}{}
    + \left( \vec{x}{\ell+1}{j_\ell} - \vec{x}{\ell}{j_\ell} \right)
    e_{j_\ell}
    \right) = \vec{z}{\ell}{} + A_{:,j_\ell} \left( \vec{x}{\ell+1}{j_\ell}
    - \vec{x}{\ell}{j_\ell} \right).
\end{equation*}
Therefore, updating $\vec{z}{\ell+1}{}$ from $\vec{z}{\ell}{}$ costs
$\bigO(n)$ or less operations. Generally, most coordinate-wise descent
methods cost $\bigO(n)$ or less operations per iteration, which is much
smaller than traditional iterative method with $\bigO(n^2)$ operations
per iteration. Such a gap of the computational cost per iteration enables
CDM focusing on the update of more important coordinates throughout
iterations. The increase of the number of iterations is also leveraged by
the choice of stepsize in the updating. The upper bound for the stepsize
with guaranteed convergence in the CDM could be much larger than that
in the traditional gradient descent method. Therefore, although CDM
requires a little more number of iterations to achieve the convergence
criteria, the operations counts and the runtime is much smaller than
that of many traditional iterative methods~\cite{Chow2017, Lei2016}.

\begin{table}[htp]
    \centering
    \begin{tabular}{ccm{10cm}}
        \toprule
        & Short name & Explanation\\
        \toprule
        \multirow{3}{*}{Type of CDM} & CD & plain CDM \\
        & GCD & greedy CDM, the coordinate is pick via a greedy way \\
        & SCD & stochastic CDM, the coordinate is sampled from a
        probability distribution \\
        \toprule
        \multirow{4}{*}{Coordinate-pick} & Cyc & the coordinate is
        chosen in a cyclic way \\
        & Uni & the coordinate is sampled uniformly \\
        & Grad & the coordinate is chosen according to the magnitude of
        the gradient\\
        & LS & the coordinate is chosen according to the exact
        line-search \\
        \toprule
        \multirow{3}{*}{Coordinate-update} & Grad & the coordinate is
        updated according to the gradient multiplying a stepsize\\
        & LS & the coordinate is updated according to the exact
        line-search \\
        & vecLS & the exact line-search is applied to a sparse vector
        direction \\
        \bottomrule
    \end{tabular}
    \caption{Short names in name convention.}
    \label{tab:shortname}
\end{table}

Throughout this paper, several proposed CDMs together with
existing methods will be mentioned many times. In order to reduce
the difficulty in remembering all these names, we will follow
a systematic name convention. The name of a CDM is composed of
three parts, the type of coordinate-wise descent, coordinate-pick
strategy, and coordinate-update strategy. Three parts are separated
with hyphens. Table~\ref{tab:shortname} defines the short names used in
each part. Some of the short name works with specific choice of the type
of CDM.  For example, Uni can only be combined with SCD.  We remark that
some of the existing methods in literature are renamed under this system,
e.g., the coordinate descent method considered in~\cite{Desopo1959,
Hong2017} and \cite{Lee2017a} are called CD-Cyc-LS and CD-Cyc-Grad
respectively; SGCD~\cite{Lei2016} is renamed as GCD-Grad-LS; etc.

\subsection{Contribution}
We first highlight our contribution as following bullet points and then
discuss in detail.
\begin{itemize}
    \item Analyze the landscape of $f(x)$ as
        \eqref{eq:LEVP-opt} in detail;
    \item Derive GCD-LS-LS as the most greedy CDM of $f(x)$ and show
        that most of the saddle points can be avoided under the method;
    \item Propose SCD-Grad-vecLS($t$) and SCD-Grad-LS($t$) as a family
        of stochastic CDMs of $f(x)$ for $t$ being the sampling power,
        and the local convergence is proved with a convergence rate
        monotonically increasing as $t$ increases.
\end{itemize}

In more details, we first analyze the landscape of the objective
function $f(x)$. Through our analysis, we show that, although $f(x)$
is non-convex, $x = \pm \sqrt{\lambda_1} v_1$ are the only two local
minima of $f(x)$. Since they are of the same function value, we conclude
that all local minima of $f(x)$ are global minima.

Then, we investigate a gradient based CDM, CD-Cyc-Grad.  It selects
coordinate in a cyclic way and the updating follows the gradient vector
restricted to that coordinate multiplying by a fixed stepsize. Thanks
to the locality of the gradient updating, we show that CD-Cyc-Grad
converges to global minima almost surely for the LEVP optimization
problem \eqref{eq:LEVP-opt}.

As many other problems solved by CDMs, exact line search along each
coordinate direction can be conducted for \eqref{eq:LEVP-opt}. We further
derive that maximum coordinate improvement is achievable in $\bigO(n)$
operations, which leads to a CDM called GCD-LS-LS. Through the analysis
of saddle points and the greedy strategy in GCD-LS-LS, we find that
many saddle points of the non-convex problem in \eqref{eq:LEVP-opt}
are escapable.

Though greedy method, the GCD-LS-LS, is efficient with single coordinate
update per-iteration, they often fail in convergence when multiple
coordinates are updated per-iteration. Such a problem can be resolved
by introducing stochastic coordinate sampling. The SCD-Grad-vecLS($t$)
and the SCD-Grad-LS($t$) sample several coordinates per-iteration with
the probabilities proportional to the $t$-th power of the gradient vector
at current iteration. When the power goes to infinity, the stochastic
CDMs become the greedy ones. Further, we analyze the local convergence
property for the stochastic CDMs for different sampling power $t$.
The theorem can be applied to show the local convergent property of
GCD-Grad-LS and GCD-LS-LS as corollaries.  One important message of
the theorem is that larger sampling power $t$ leads to faster local
convergence rate. Therefore the convergence rate of either GCD-LS-LS
or SCD-Grad-LS($t$) with $t>0$ is provably faster or equal to that of
SCD-Grad-LS($0$) which corresponds to the uniform sampling.  However,
greedy CDMs or stochastic CDMs with larger $t$ are more difficult to
escape from strict saddle points when the objective is non-convex.
Therefore, through our analysis and numerical experiments, we recommend
SCD-Grad-LS($1$) for LEVP.

Although all methods are introduced as a solution to $f(x)$
in \eqref{eq:LEVP-opt}, they, especially SCD-Grad-vecLS($t$) and
SCD-Grad-LS($t$), can be widely extended to other problems. Most of
the associated analysis could be extended as well.

All proposed and reviewed methods are tested on synthetic matrices and
eigenvalue problems from quantum many-body problems.  In all examples,
CDMs of \eqref{eq:LEVP-opt} outperform power method, coordinate-wise
power method, and full gradient descent with exact line search. When the
matrix is shifted by a big positive number, CDMs of \eqref{eq:LEVP-opt}
converges in the similar number of iterations as in the case without
shifting. While power method or coordinate-wise power method, which are
sensitive to the shifting, converge much slower. For the matrix from
quantum many-body problems, where we know a priori that some of the
coordinates of the leading eigenvector is more important than others due
to the physical property, all CDMs including CDMs of \eqref{eq:LEVP-opt}
and coordinate-wise power method significantly outperform full vector
updating methods. This shows great potential of applying the CDMs to
quantum many-body problems.

\subsection{Organization}

The rest of the paper is organized as follows. We introduce
notations and analyze the landscape of \eqref{eq:LEVP-opt} in
Section~\ref{sec:analysis}. Section~\ref{sec:CDM}, \ref{sec:GCDM} and
\ref{sec:SCDM} follow the same structure, which first introduces or
reviews several CDMs, and then conducts the corresponding analysis.
Section~\ref{sec:CDM} focuses on CDMs with conservative stepsize,
whereas Section~\ref{sec:GCDM} and Section~\ref{sec:SCDM} focus on
greedy and stochastic CDMs respectively.  The numerical results are
given in Section~\ref{sec:num-res}, followed by the conclusion and
discussion in Section~\ref{sec:conclusion}.

\section{Landscape Analysis}
\label{sec:analysis}

This section focuses on the analysis of the variational problem
\eqref{eq:LEVP-opt}. More specifically, we analyze the landscape of
$f(x)$, especially properties of stationary points in this section. The
results show some advantages in working with \eqref{eq:LEVP-opt}, and
more importantly, provide insights in designing optimization algorithms.

\smallskip 
\noindent\emph{Notations.}
Before the detailed analysis, we define a few common notations as in
Table~\ref{tab:notation}. These notations will be used without further
explanation for the rest of the paper.

\begin{table}[htp]
    \centering
    \begin{tabular}{ll}
        \toprule
        Notation & Explanation \\
        \toprule
        $n$    & the size of the problem \\
        $i,j$  & coordinate index \\
        $\Omega$  & a set of coordinate index \\
        $k$  & the size of $\Omega$ \\
        $\ell$ & the iteration index\\
        $\alpha,\beta$ & real coefficients\\
        $\lambda, \lambda_1$ & an eigenvalue and the largest eigenvalue\\
        $v, v_1$ & an eigenvector and the eigenvector associated with
        $\lambda_1$ \\
        $e_i$ & indicator vector with one on the $i$-th entry and zero
        everywhere else \\
        $x, y, z$ & vectors \\
        $\vec{x}{\ell}{}$ & the vector of the $\ell$-th iteration \\
        $x_i$ & the $i$-th entry of vector $x$ \\
        $x_\Omega$ & a vector with entries indexed by $\Omega$ of $x$ \\
        $A$ & the matrix under consideration \\
        $I$ & identity matrix, the size may depend on the context \\
        $A_{i,j}$ & the $(i,j)$-th entry of matrix $A$ \\
        $A_{i,:}, A_{:,j}$ & the $i$-th row and the $j$-th column of
        matrix $A$ \\
        $\norm{\cdot}_F,\norm{\cdot}_2$ & Frobenius norm and 2-norm \\
        \bottomrule
    \end{tabular}
    \caption{Common notations.}
    \label{tab:notation}
\end{table}
Some of the notations in Table~\ref{tab:notation} can be used in a
combined way, e.g., $\vec{x}{\ell}{i}$ denotes the $i$-th entry of the
vector $x$ at $\ell$-th iteration. The set of coordinate index,
$\Omega$, can also be applied to the subscript of a matrix, e.g.,
$A_{:,\Omega}$ denotes the columns of $A$ with index in
$\Omega$. Notice that Frobenius norm and 2-norm are different measure
for a matrix, but they are the same measure acting on a
vector. Therefore, we will drop the subscript of the norm for vectors,
i.e., $\norm{x} = \norm{x}_2 = \norm{x}_F$.

\medskip 
\noindent 
\emph{Stationary points and global minimizers.}
When working with non-convex problems, the understanding of the
landscape of the objective function is crucial for iterative methods.
If the objective function is bounded from below and the gradient
of the objective function is assumed to be Lipschitz continuous,
gradient based iterative methods generally are guaranteed to converge
to a stationary point such that $\norm{\grad f(x)} < \epsilon$ within
$\bigO(1/\epsilon^2)$ iterations~\cite{Nesterov2004}.  Without rate,
\citet{Lee2017a} show that gradient based iterative methods converge to
local minima almost surely if all saddle points are strict saddle points.
Adding a random perturbation to the gradient based iterative methods
enables the analysis of the convergence to local minima with various
rates~\cite{Ge2015,Jin2017}. However, convergence to global minima in
most cases is not guaranteed if the landscape of the objective function
is complicated.

In this section, we analyze the landscape of \eqref{eq:LEVP-opt}
and show that every local minimum is a global minimum. For
simplicity of presentation, we assume that $f(x)$ is a second order
differentiable function (which obviously holds when $f$ is given by
\eqref{eq:LEVP-opt}).  Denote the gradient vector and Hessian matrix
of $f(x)$ as $\grad f(x)$ and $\grad^2 f(x)$ respectively. We give the
following definitions.

\begin{definition}[Stationary point] \label{def:stationary-point}
    A point $y$ is a stationary point of $f$ if $\grad f(y) = 0$.
\end{definition}

\begin{definition}[Strict saddle point]
    \label{def:strict-saddle-point}
    A point $y$ is a strict saddle point of $f$ if $y$ is a stationary
    point and $\lambda_{\text{min}} \left( \grad^2 f(y) \right) <
    0$, where $\lambda_{\text{min}} \left( \grad^2 f(y) \right)$
    is the smallest eigenvalue of the Hessian matrix.\footnote{This
    definition of the strict saddle point includes local maximizer
    as well, which does not cause trouble in minimizing $f(x)$ as in
    \eqref{eq:LEVP-opt}.}
\end{definition}

\begin{definition}[Local minimizer]
    \label{def:local-minimizer}
    A point $y$ is a local minimizer of $f$ if there exists an $\epsilon$
    such that $f(y) \leq f(x)$ for any $\norm{x-y} \leq \epsilon$. If
    further $f(y) < f(x)$ when $x \neq y$, then $y$ is called a strict
    local minimizer.
\end{definition}

Following these definitions, we explicitly write down the form of the
stationary points, strict saddle points and local minimizers of the
objective function $f(x)$ in \eqref{eq:LEVP-opt}. The assumptions on
the matrix $A$ are summarized in Assumption~\ref{ass:matA}.

\begin{assumption}
    \label{ass:matA}
    The matrix $A$ is symmetric with eigenvalues and eigenvectors
    $\lambda_1 > \lambda_2 \geq \lambda_3 \geq \cdots \geq \lambda_n$
    and $v_1,v_2,v_3,\dots,v_n$ respectively. In addition, the largest
    eigenvalue is positive, $\lambda_1 > 0$.
\end{assumption}

\begin{lemma}
    \label{lem:stationary-point}
    Under Assumption~\ref{ass:matA}, $x=0$ is a stationary point
    of $f(x)$. Other stationary points of $f(x)$ are of the form
    $\sqrt{\lambda}v$ where $\lambda$ is a positive eigenvalue of $A$ and
    $v \in \nullsp{\lambda I - A}$ with $\norm{v} = 1$. In particular,
    $ \pm \sqrt{\lambda_1} v_1$ are both stationary points.
\end{lemma}

\begin{proof}
    The stationary point of $f(x)$ satisfies
    \begin{equation} \label{eq:gradf}
        \grad f(x) = -4 A x + 4 \left( x^\top x \right) x = 0.
    \end{equation}
    Obviously, $x = 0$ is a stationary point. When $x \neq 0$, we have
    $\left( x^\top x I - A \right) x = 0$. This implies that $x^\top
    x I - A$ is singular, i.e., $x^\top x = \lambda$ for $\lambda$
    being a positive eigenvalue of $A$. When $\lambda = \lambda_1$,
    we have $x \in \nullsp{\lambda_1 I - A} = \spansp{v_1}$ and $x^\top
    x = \lambda_1$, which implies $x = \pm \sqrt{\lambda_1} v_1$. When
    $0 < \lambda < \lambda_1$, we have $x \in \nullsp{\lambda I - A}$
    and $x^\top x = \lambda$. Hence, $x = \sqrt{\lambda} v$ for $v \in
    \nullsp{\lambda I - A}$ and $\norm{v} = 1$.
\end{proof}

\begin{lemma} \label{lem:strict-saddle-point}
    Under Assumption~\ref{ass:matA}, $0$ and $\sqrt{\lambda} v$ with $v
    \in \nullsp{\lambda I - A}$ and $\norm{v} = 1$ are strict saddle
    points of $f(x)$, where $\lambda < \lambda_1$ and $\lambda$ is a
    positive eigenvalue of $A$.
\end{lemma}

\begin{proof}
    According to Lemma~\ref{lem:stationary-point}, $0$ and
    $\sqrt{\lambda} v$ are stationary points of $f(x)$.
    It suffices to validate the second condition in
    Definition~\ref{def:strict-saddle-point}, i.e.,
    \begin{equation} \label{eq:min-lambda}
        \lambda_{min} \left( \grad^2 f(x) \right) = \lambda_{min} \left(
        -4 A + 8 xx^\top + 4 x^\top x I \right) < 0.
    \end{equation}
    When $x = 0$, \eqref{eq:min-lambda} is obvious since $\lambda_1 >
    0$. When $x = \sqrt{\lambda} v$, we apply $\grad^2 f(x)$ to $v_1$
    and get
    \begin{equation}
        \grad^2 f(x) v_1 = 4 (\lambda - \lambda_1) v_1,
    \end{equation}
    where we have used the orthogonality between $v$ and $v_1$. Therefore
    $4(\lambda - \lambda_1)$ is a negative eigenvalue of $\grad^2 f(x)$
    and $\sqrt{\lambda} v$ is a strict saddle point.
\end{proof}

From the proof, we observe that for all the strict saddle points
of $f$, the leading eigenvector $v_1$ of $A$ is always an unstable
direction. This will help us in the convergence proof.

The following theorem follows directly from
Lemmas~\ref{lem:stationary-point} and \ref{lem:strict-saddle-point}.

\begin{theorem} \label{thm:global-minimizer}
    Under Assumption~\ref{ass:matA}, local minimizers of $f(x)$ are
    given by $\pm \sqrt{\lambda_1} v_1$ and both local minimizers of
    $f(x)$ are global minimizers.
\end{theorem}

\begin{proof}
    Lemma~\ref{lem:stationary-point} and \ref{lem:strict-saddle-point}
    imply that the only possible local minimizers are $\sqrt{\lambda_1}
    v_1$ and $-\sqrt{\lambda_1} v_1$. We easily check that the
    objective function values are equal at both $\sqrt{\lambda_1} v_1$
    and $-\sqrt{\lambda_1} v_1$:
    \begin{equation}
        f(\sqrt{\lambda_1} v_1) = f(-\sqrt{\lambda_1} v_1) = \norm{A -
        \lambda_1 v_1 v_1^\top }_F^2.
    \end{equation}
    Because $f(x) \geq 0$ is bounded from below and $\abs{f(x)}
    \to \infty$ as $x\to\infty$, global minimizers exist. Therefore,
    $\pm \sqrt{\lambda_1} v_1$ are both local minimizers as well as
    global minimizers.
\end{proof}

Theorem~\ref{thm:global-minimizer} shows that $f(x)$ has no spurious
local minima. If an iterative method converges to a local minimum of
$f(x)$, it achieves the global minimum. The remaining obstacle of the
global convergence is the (strict) saddle points.

\section{Coordinate-wise descent method with conservative stepsize}
\label{sec:CDM}

In this section, we discuss CD-Cyc-Grad, whose ``coordinate-update'' is
as the entries of the gradient multiplied by a stepsize.  The method has
guarantee that for almost all initial values, the iterative procedure
converges to the global minimum. While, the choice of stepsize
is problem-dependent and conservative.  According to our numerical
experiments, the number of iterations of these coordinate-wise gradient
based methods are too large to be competitive.

\begin{algorithm}[htpb]
    \caption{CD-Cyc-Grad for LEVP}
    \label{alg:CD-Cyc-Grad}

    Input: Symmetric matrix $A \in \bbR^{n \times n}$; initial vector
    $\vec{x}{0}{}$; stepsize $\gamma$.

\begin{algorithmic}[1]

    \STATE $\vec{z}{0}{} = A\vec{x}{0}{}$

    \STATE $\ell = 0$

    \WHILE {(not converged)}
        \STATE $j_{\ell} = ( \ell \mod n ) + 1$

        \STATE $\vec{x}{\ell+1}{j} = \left\{
            \begin{array}{ll}
                \vec{x}{\ell}{j} - \gamma \left( -4\vec{z}{\ell}{j} +
                4 \norm{\vec{x}{\ell}{}}^2 \vec{x}{\ell}{j}\right), &
                j=j_{\ell}\\ \vec{x}{\ell}{j}, & j \neq j_{\ell}
            \end{array}
            \right.$

        \STATE $\vec{z}{\ell+1}{} = \vec{z}{\ell}{} + A_{:,j_{\ell}}\left(
        \vec{x}{\ell+1}{j_{\ell}} - \vec{x}{\ell}{j_{\ell}} \right)$

        \STATE $\ell = \ell + 1$
    \ENDWHILE
\end{algorithmic}

\end{algorithm}

\subsection{CD-Cyc-Grad and SCD-Cyc-Grad}

The first coordinate-wise descent method we consider addressing
\eqref{eq:LEVP-opt} is CD-Cyc-Grad, following the name convention in
Section~\ref{sec:introCDM}.

CD-Cyc-Grad conducts the ``coordinate-pick'' step in
Algorithm~\ref{alg:CDM} in a cyclic way, i.e.,
\begin{equation}
    j_{\ell} = (\ell \mod n) + 1.
\end{equation}
Hence, all coordinates are picked in a fixed order with almost equal
number of updatings throughout iterations.  The ``coordinate-update''
adopts coordinate-wise gradient, that is
\begin{equation} \label{eq:grad-update}
    \vec{x}{\ell+1}{j_{\ell}} = \vec{x}{\ell}{j_{\ell}} - \gamma
    \grad_{j_{\ell}} f(\vec{x}{\ell}{}) = \vec{x}{\ell}{j_{\ell}}
    - \gamma \Bigl( -4A_{j_\ell,:} \vec{x}{\ell}{} + 4
    \norm{\vec{x}{\ell}{}}^2 \vec{x}{\ell}{j_\ell} \Bigr),
\end{equation}
where $\grad_{j_{\ell}}f(\vec{x}{\ell}{})$ denotes the $j_{\ell}$-th
entry of the gradient of $f$ at $\vec{x}{\ell}{}$ and $\gamma$ is the
stepsize.  CD-Cyc-Grad is detailed as Algorithm~\ref{alg:CD-Cyc-Grad},
with $\vec{z}{\ell}{} = A\vec{x}{\ell}{}$ being adopted in
\eqref{eq:grad-update}. The advantage of CD-Cyc-Grad over the full
gradient descent (GD) method is mainly that the choice of the stepsize
in CD-Cyc-Grad could be much larger than that in GD, which could lead
to faster convergence~\cite{Nesterov2012}.

Here we mention another widely used choice of the ``coordinate-pick''
called random cyclic, which involves randomness. In the beginning
of every $n$ iterations, a random permutation $\Pi^{(\ell)}$ of the
indices $1,\dots,n$ is generated. $\Pi^{(\ell)}$ is a vector of size $n$
and the superscript $\ell$ denotes the iteration number when the random
permutation is generated. Once the random permutation is provided, the
following $n$ iteration update the coordinate according to the order in
$\Pi^{(\ell)}$. Recently, several works~\cite{Gurbuzbalaban2017, Lee2016,
Sun2016} discussed the comparison of the convergence rates for CDMs with
cyclic and random strategy of ``coordinate-pick'' for convex problems.

\subsection{Global convergence of gradient based coordinate-wise descent
method}

In this section, we show the global convergence of CD-Cyc-Grad
in Theorems~\ref{thm:CD-Cyc-Grad}. In this paper, global convergence means the iterate points converge to the global minimum as the number of iterations goes to infinity. For simplicity of the argument,
here we further restrict the assumption of $A$ such that the positive
eigenvalues of $A$ are distinct.

The following theorem establish the global convergence (up to measure
$0$ set of initial conditions) of the Algorithms~\ref{alg:CD-Cyc-Grad}
without convergence rate.

\begin{theorem} \label{thm:CD-Cyc-Grad} Let
    $R \geq \sqrt{\max_j \norm{A_{:,j}}}$ be a constant and $\gamma
    \leq \frac{1}{4(n+4)R^2}$ be the stepsize.  Assume the iteration
    follows CD-Cyc-Grad as in Algorithm~\ref{alg:CD-Cyc-Grad} and the
    iteration starts from the domain $W_0 = \left\{ x: \norm{x}_\infty <
    R \right\}$.  The iteration converges to global minima for all
    $\vec{x}{0}{}
    \in W_0$ up to a measure zero set.
\end{theorem}

The proof of the theorem can be found in
Appendix~\ref{app:CD-Cyc-Grad}. The idea of the proof follows the
recent work \cite{Lee2017a}.  Comparing the choice of stepsize in
Theorem~\ref{thm:CD-Cyc-Grad} with the CDM stepsize for some other
optimization problems, for example~\cite{Richtarik2014, Wright2015},
the stepsize $\gamma$ here is about a fraction of $1/n$ smaller. This
is due to the fact that the diagonal of the Hessian of $f(x)$, as in
\eqref{eq:min-lambda}, is unbounded from above.  The choice of $R$
and $\gamma$ ensures that the iteration stays within $W_0$, in which
the Hessian of $f$ remains bounded. It is worth pointing out that,
according to our numerical experiments, when $\vec{x}{0}{}$ is set to
be close to the boundary of $W_0$, larger $\gamma$ leads to divergent
iteration. When $\vec{x}{0}{}$ is set close to the origin, the stepsize
$\gamma$ can be tuned sightly larger for the our testing cases with some
randomly generated matrix $A$. Although the choice of $\gamma$ is very
restrictive and leads to slow convergence, CD-Cyc-Grad has a guarantee
of global convergence up to a measure zero set of initial points.

\section{Greedy coordinate-wise descent method}
\label{sec:GCDM}

In this section, we will review a greedy CDM, GCD-Grad-LS, and present
another fully greedy CDM, named GCD-LS-LS. Both greedy methods update
the coordinate according to exact line search, and thus the stepsize
could be much larger than the conservative choice in CD-Cyc-Grad.
GCD-Grad-LS selects coordinate according to the magnitude of the gradient
vector and then performs the exact line search along that coordinate,
while GCD-LS-LS conducts an exact line search along all coordinate
directions and move to the minimizer.  We show the advantage of the
exact line search in GCD-LS-LS, as it can escape certain saddle points
of $f$, which will be discussed in Section 4.2.

\subsection{GCD-Grad-LS and GCD-LS-LS}\label{sec:GCD-Grad-LS}

We first review GCD-Grad-LS, proposed in~\cite{Lei2016}. One of the
most widely used greedy strategy in the ``coordinate-pick'' is the
Gauss-Southwell rule~\cite{Southwell2010},
\begin{equation}
    j_\ell = \argmax_j \abs{ \grad_j \func{f}{ \vec{x}{\ell}{} }},
\end{equation}
which is the ``coordinate-pick'' strategy in GCD-Grad-LS. Since the
Gauss-Southwell rule selects the coordinate according to the magnitude
of the gradient vector, we denote such strategy as ``Grad'' under the
name convention.

Once the coordinate is selected, we solve the minimization problem for
the exact line search,
\begin{equation} \label{eq:cubicopt}
    \alpha_{j_\ell} = \argmin_{\alpha} f\bigl( \vec{x}{\ell}{} + \alpha
    e_{j_\ell} \bigr).
\end{equation}
Since $h(\alpha) = f\bigl( \vec{x}{\ell}{} + \alpha e_{j_\ell} \bigr)$
is a quartic polynomial in $\alpha$, solving the minimization problem
is equivalent to find the roots of $h'(\alpha) = 0$. Straightforward
calculation shows that,
\begin{equation} \label{eq:cubicroot}
    h'(\alpha) = 4 \bigl( \alpha^3 + b_{j_\ell} \alpha^2 + c_{j_\ell}
    \alpha + d_{j_\ell} \bigr) = 0,
\end{equation}
where the coefficients are defined as 
\begin{equation}
    b_{j_\ell} = 3\vec{x}{\ell}{j_\ell}, \quad
    c_{j_\ell} = \bigl\lVert \vec{x}{\ell}{}\bigr\rVert^2 +
    2\bigl(\vec{x}{\ell}{j_\ell}\bigr)^2
    - A_{j_\ell,j_\ell}, \quad
    \text{and} \quad 
    d_{j_\ell} =
    \bigl\lVert\vec{x}{\ell}{}\bigr\rVert^2\vec{x}{\ell}{j_\ell} -
    \vec{z}{\ell}{j_\ell}.
\end{equation}
We notice that the coefficients of the cubic polynomial requires
$\bigO(n)$ operations to compute. Once the coefficients are
calculated, solving \eqref{eq:cubicroot} can be done in $\bigO(1)$
operations.  There could be multiple roots of \eqref{eq:cubicroot},
and we can use the following \emph{root picking strategy} to find 
the one minimizing $h(\alpha)$:
\begin{itemize}
    \item If there is only one root of \eqref{eq:cubicroot}, then it
    minimizes $h(\alpha)$.
    \item If there are two roots of $\eqref{eq:cubicroot}$, according
    to the property of cubic polynomial, one of them must be single
    root and the other one is of multiplicity two. The single root
    minimizes $h(\alpha)$.
    \item If there are three roots in a row, the middle one is a local
    maximizer of $h(\alpha)$ and the root further away from the middle
    one minimizes $h(\alpha)$.
\end{itemize}
Therefore, with this \emph{root picking strategy} to find the minimizer
of $h(\alpha)$, the solution of~\eqref{eq:cubicopt} can be achieved in
$\bigO(n)$ operations.

\begin{algorithm}[htp]
    \caption{GCD-Grad-LS for LEVP}
    \label{alg:GCD-Grad-LS}

    Input: Symmetric matrix $A \in \bbR^{n \times n}$; initial vector
    $\vec{x}{0}{}$.

\begin{algorithmic}[1]

    \STATE $\vec{z}{0}{} = A\vec{x}{0}{}$

    \STATE $\ell = 0$
    
    \WHILE {(not converged)}

        \STATE $\nu = \norm{\vec{x}{\ell}{}}^2$

        \STATE $j_\ell = \argmax_j \abs{ \nu \vec{x}{\ell}{j} -
        \vec{z}{\ell}{j}}$

        \STATE $b_{j_\ell} = 3\vec{x}{\ell}{j_\ell}$

        \STATE $c_{j_\ell} = \nu + 2\bigl( \vec{x}{\ell}{j_\ell} \bigr)^2 -
        A_{j_\ell,j_\ell}$

        \STATE $d_{j_\ell} = \nu \vec{x}{\ell}{j_\ell} - \vec{z}{\ell}{j_\ell}$

        \STATE Solve $\alpha^3 + b_{j_\ell} \alpha^2 + c_{j_\ell} \alpha
        + d_{j_\ell} = 0$ with the \emph{root picking strategy} for
        $\alpha_{j_\ell}$

        \STATE $\vec{x}{\ell+1}{j} = \begin{cases}
                \vec{x}{\ell}{j} + \alpha_{j_\ell}, & j=j_\ell\\
                \vec{x}{\ell}{j}, & j \neq j_\ell
              \end{cases}$
            
        \STATE $\vec{z}{\ell+1}{} = \vec{z}{\ell}{} +
        A_{:,j_\ell}\alpha_{j_\ell}$

        \STATE $\ell = \ell + 1$

    \ENDWHILE
\end{algorithmic}

\end{algorithm}

Algorithm~\ref{alg:GCD-Grad-LS} describes the steps of GCD-Grad-LS
in detail. The local convergence of GCD-Grad-LS can be established:
GCD-Grad-LS can be viewed as a special case of SCD-Grad-LS($t$)
with $t=\infty$ and $k=1$, which will be discussed later in
Section~\ref{sec:local-conv}; We will prove the local convergence
of SCD-Grad-LS for all $t \geq 0$, where the local convergence of
GCD-Grad-LS is automatically implied.

We observe that in \eqref{eq:cubicroot}, all the $\bigO(n)$ computational
cost comes from the calculation of $\norm{\vec{x}{\ell}{}}^2$. Once
$\nu = \norm{\vec{x}{\ell}{}}^2$ is pre-calculated, all coefficients in
\eqref{eq:cubicroot} can be calculated in $\bigO(1)$ operations and hence
\eqref{eq:cubicopt} can be solved in $\bigO(1)$ operations. Therefore,
applying this to each coordinate, once $\nu$ is pre-calculated, solving
\eqref{eq:cubicopt} for all coordinates can be done in $\bigO(n)$
operations. This leads us to investigate the possibility of conducting
exact line search along all coordinates,
\begin{equation} \label{eq:greedyLS-opt}
    \begin{split}
        & j_\ell,\alpha_{j_\ell} = \argmin_{j,\alpha}
        f\left( \vec{x}{\ell}{}
        + \alpha e_j \right) \\
        \iff &
        j_\ell = \argmin_{j} \func{f}{ \vec{x}{\ell}{} + \alpha_j e_j }
        \quad \mathrm{with} \quad \alpha_j = \argmin_\alpha \func{f}{
        \vec{x}{\ell}{} + \alpha e_j }.
    \end{split}
\end{equation}
Based on the discussion above, $\alpha_j$ for all $j=1,\dots,n$ can be
obtained in $\bigO(n)$ operations. Now, we will show that evaluating
the difference of $\func{f}{ \vec{x}{\ell}{} + \alpha_j e_j }$ and
$\func{f}{ \vec{x}{\ell}{} }$ for all $j$ can be done in $\bigO(n)$
operations as well. Therefore the minimization problem of $j$ in the
second line of \eqref{eq:greedyLS-opt} can be solved efficiently.
Through tedious but straightforward calculation, we obtain
\begin{equation} \label{eq:deltaf}
    \Delta f_j := \func{f}{ \vec{x}{\ell}{} + \alpha_j e_j } - \func{f}{
    \vec{x}{\ell}{} }  = \alpha_j^4 + \frac{4b_j}{3} \alpha_j^3 +
    2c_j \alpha_j^2 + 4d_j\alpha_j
\end{equation}
where $b_j, c_j$ and $d_j$ are defined analogous to
\eqref{eq:cubicroot}. Combining \eqref{eq:cubicroot} and
\eqref{eq:deltaf}, we conclude that \eqref{eq:greedyLS-opt} is achievable
in $\bigO(n)$ operations. The corresponding method is described in
Algorithm~\ref{alg:GCD-LS-LS}.

\begin{algorithm}[htp]
    \caption{GCD-LS-LS for LEVP}
    \label{alg:GCD-LS-LS}

    Input: Symmetric matrix $A \in \bbR^{n \times n}$; initial vector
    $\vec{x}{0}{}$.

\begin{algorithmic}[1]

    \STATE $\vec{z}{0}{} = A\vec{x}{0}{}$

    \STATE $\ell = 0$

    \WHILE {(not converged)}

        \STATE $\nu = \norm{\vec{x}{\ell}{}}^2$

        \FOR {$j = 1, 2, \dots, n$}

            \STATE $b_{j} = 3\vec{x}{\ell}{j}$

            \STATE $c_{j} = \nu + 2\bigl( \vec{x}{\ell}{j} \bigr)^2
            - A_{j,j}$

            \STATE $d_{j} = \nu \vec{x}{\ell}{j} - \vec{z}{\ell}{j}$

            \STATE Solve $\alpha^3 + b_{j} \alpha^2 + c_{j} \alpha
            + d_{j} = 0$ with the \emph{root picking strategy} for
            $\alpha_{j}$

            \STATE $ \Delta f_j = \alpha_j^4 + \frac{4b_j}{3} \alpha_j^3
            + 2c_j \alpha_j^2 + 4d_j\alpha_j$

        \ENDFOR

        \STATE $j_\ell = \argmin_j \Delta f_j$

        \STATE $\vec{x}{\ell+1}{j} = \begin{cases}
                \vec{x}{\ell}{j} + \alpha_j, & j=j_\ell\\
                \vec{x}{\ell}{j}, & j \neq j_\ell
              \end{cases}$

        \STATE $\vec{z}{\ell+1}{} = \vec{z}{\ell}{} + A_{:,j_\ell}
        \alpha_{j_\ell}$

        \STATE $\ell = \ell + 1$
    \ENDWHILE
\end{algorithmic}

\end{algorithm}

Similar to GCD-Grad-LS, the local convergent of GCD-LS-LS can be
established; we will defer the local convergence analysis to the end
of Section~\ref{sec:local-conv}.

\subsection{Escapable saddle points using exact line search}
\label{sec:escapable-saddle}

This section discusses one advantage in working with the exact line
search along all coordinate directions, \eqref{eq:greedyLS-opt}, of
escaping some of the saddle points. The iteration of GCD-LS-LS can be
summarized as
\begin{equation}\label{eq:greedy-iter} 
    \begin{aligned}
        & \vec{x}{\ell+1}{} = \vec{x}{\ell}{} + \alpha_\ell e_{j_\ell}, \\
        & \text{with } \alpha_\ell,j_\ell = \argmin_{\alpha,j} f
        \bigl(\vec{x}{\ell}{} + \alpha e_j \bigr).
    \end{aligned}
\end{equation}

\begin{theorem} \label{thm:greedy}
    Assume Assumption~\ref{ass:matA} holds and $x^s \neq 0$ is a strict
    saddle point associated with eigenvalue $\lambda < \max_i A_{i,i}$.
    There exists a constant $\delta_0$ such that if $\norm{x^{(0)} -
    x^s} < \delta_0$, then $f(x^{(1)}) < f(x^s)$.
\end{theorem}

Since exact line search in all direction guarantees that
$\func{f}{\vec{x}{\ell+1}{}} < \func{f}{\vec{x}{\ell}{}}$, the theorem
guarantees that the iteration will never come back to the neighborhood
of the saddle point.

\begin{proof}
    Without loss of generality, we assume that $A_{1,1} = \max_i
    A_{i,i}$. 

    Let $\Delta x = x - x^s$ and $\norm{\Delta x} = \delta$. We update
    $x$ by $\beta_1 e_1$.  The update of the objective function is
    given by
    \begin{equation} \label{eq:thm-decay} 
        f\left(x + \beta_1 e_1\right) = f(x^s) + C_0 + \Delta x^\top C_1 +
        \Delta x^\top C_2 \Delta x + (\Delta x^\top \Delta x) \Delta
        x^\top C_3 +
        (\Delta x^\top \Delta x)^2,
    \end{equation}
    where
    \begin{equation}
        \begin{split}
            C_0 & =  \beta_1^2 \left( \beta_1^2 + 4 x^s_1 \beta_1 +
            2\lambda + 4 \left( x^s_1 \right)^2 - 2A_{1,1} \right) =:
            q(\beta_1), \\
            C_1 &=  -4 \beta_1 A_{:,1} + 4 \beta_1 \lambda e_1 + 8
            \beta_1 x^s_1 x^s + 4 \beta_1^2 x^s + 8 \beta_1 x^s_1 e_1 +
            4 \beta_1^3 e_1, \\
            C_2 &=  -2A + 4 x^s \left(x^s\right)^\top + 4 \beta_1 e_1
            e_1^\top + 2\lambda I + 4 \beta_1 e_1 \left( x^s
            \right)^\top +
            4 \beta_1 x^s e_1^\top + 4 \beta_1 x^s_1 I + 2 \beta_1^2 I, \\
            C_3 &=  4 x^s + 4 \beta_1 e_1.
        \end{split}
    \end{equation}
    Applying Cauchy-Schwartz inequality to the last four terms in
    \eqref{eq:thm-decay}, we have
    \begin{equation}
        f(x+\beta_1e_1) \leq f(x^s) + q(\beta_1) +
        \norm{C_1}\delta + \norm{C_2}_2\delta^2 + \norm{C_3}\delta^3 +
        \delta^4 = f(x^s) + p(\delta),
    \end{equation}
    where $p(\delta) = q(\beta_1) + \norm{C_1}\delta +
    \norm{C_2}_2\delta^2 + \norm{C_3}\delta^3 + \delta^4$ is a quartic
    polynomial in $\delta$ with constant coefficient (independent of
    $\Delta x$).  Observe that for the polynomial in $\beta$
    \begin{equation}
        q(\beta) = \beta^2 \bigl( \beta^2 + 4 x^s_1 \beta + 2\lambda +
        4 \left( x^s_1 \right)^2 - 2A_{1,1} \bigr),
    \end{equation}
    $\beta^2$ is always positive, and the discriminant of the remaining
    $2$-nd order polynomial is $\Delta = 8(A_{1,1}-\lambda) > 0$ by
    assumption of $\lambda$. Thus we can choose $\beta_1$ such that
    $q(\beta_1) < 0$, and hence $p(\delta = 0) = q(\beta_1) < 0$.

    By continuity, there exists $\delta_0 > 0$ such that $\forall 0
    \leq \delta < \delta_0$, $p(\delta) < 0$.  Hence, following the
    greedy coordinate-wise iteration as in \eqref{eq:greedy-iter}
    and $\vec{x}{0}{} = x$ for $\norm{x-x^s} \leq \delta < \delta_0$,
    we obtain,
    \begin{equation}
        f\bigl(\vec{x}{1}{}\bigr) = f\bigl( \vec{x}{0}{} + \alpha_\ell
        e_{j_\ell}\bigr) \leq f\bigl( \vec{x}{0}{} + \beta_1 e_1\bigr)
        < f(x^s).
    \end{equation}
    The iteration escapes the strict saddle point $x^s$ in one step.
\end{proof}

To authors' best knowledge, the analysis of global convergence of greedy
coordinate-wise descent method is still open for non-convex objective
function.  Theorem~\ref{thm:greedy} provides more insights of the
behavior of the methods around the saddle points of \eqref{eq:LEVP-opt}.
The number of problematic saddle points can be very limited or even zero
for a given matrix in practice. For example, combining the result with
the Gershgorin circle theorem to locate the eigenvalues can rule out
many saddle points. In the literature, there are some
analysis of the convergence of coordinate-wise descent method, such as
\cite{Tseng2001}. However, the analysis~\cite{Tseng2001} shows the
convergence to coordinate-wise local minima but not the general local
minima.

\section{Stochastic coordinate-wise descent methods}
\label{sec:SCDM}

Greedy coordinate-wise descent method, including both GCD-Grad-LS and
GCD-LS-LS in Section~\ref{sec:GCDM}, updates a single coordinate every
iteration. These methods are beautiful from the theoretical point
of view.  In practice, these single coordinate methods are not
satisfactory on modern computer architecture. Distributed memory super
computing cluster, shared big memory machine or even personal laptop
have multi-thread parallelism enabled. In order to fully use the
computing resources, multi-coordinate updating per iteration is
desired for practical usage so that each thread can process a single
coordinate simultaneously. The direct extension of the greedy methods
to multi-coordinate version simply replaces ``coordinate-pick'' from
the most desired coordinate to the $k$ most desired coordinates; and
the ``coordinate-update'' remains the same for each picked coordinate.
However, this change leads to non-convergent iteration.
Figure~\ref{fig:gcdm-vs-scdm} (a) demonstrates the non-convergence
behavior of the greedy CDMs with $k=4$.  In this section, we propose
another natural extension of the greedy methods, i.e., stochastically
sampling multiple coordinates from certain probability distribution and
updating each coordinate accordingly.  We also provide the analysis of
the method for (locally) strongly convex problems. The local
convergence analysis of SCD-Uni-LS and GCD-Grad-LS can be viewed as
two extreme cases of our analysis. Moreover, the local convergence of
GCD-LS-LS also follows as a direct corollary.  As will be shown in the
analysis part, greedy CDMs are provably outperforms stochastic CDMs
when the initial value is in a strongly convex region near the
minimizers and single coordinate updating is adopted. If the iteration
starts from a non-convex region, as shown in
Figure~\ref{fig:gcdm-vs-scdm} (b), stochastic CDMs could converge
faster than greedy CDMs. The behavior in the figure will be discussed
in more detail at the end of this section.

\begin{figure}[htp]
    \centering
    \begin{subfigure}[t]{0.48\textwidth}
        \includegraphics[width=\textwidth]{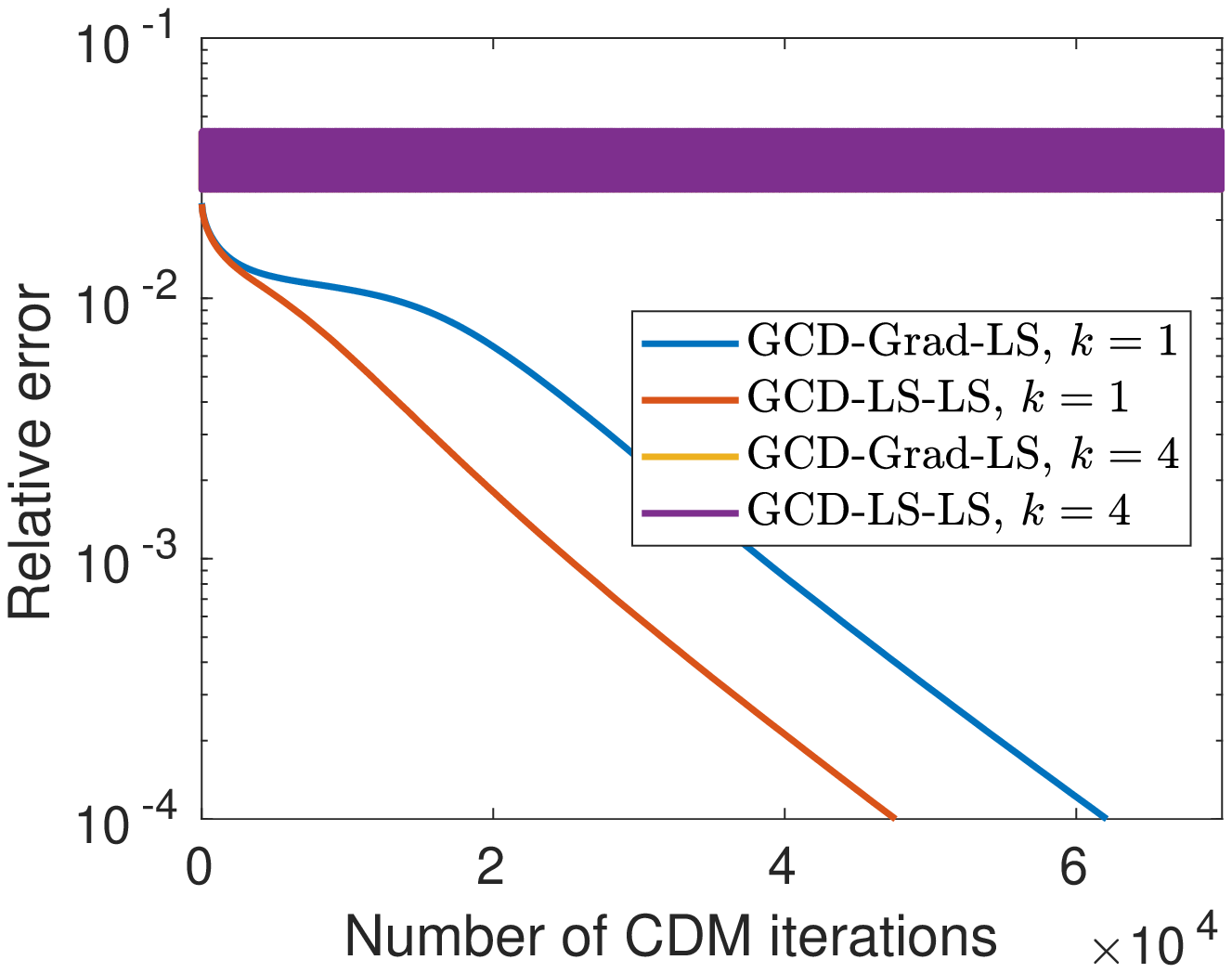}
        \caption{The non-convergence behavior of greedy CDMs on
          $A_{108}$ with $k=4$ comparing with convergent greedy CDMs
          with $k=1$. The initial vector is $e_1$.}
    \end{subfigure}
    ~
    \begin{subfigure}[t]{0.48\textwidth}
        \includegraphics[width=\textwidth]{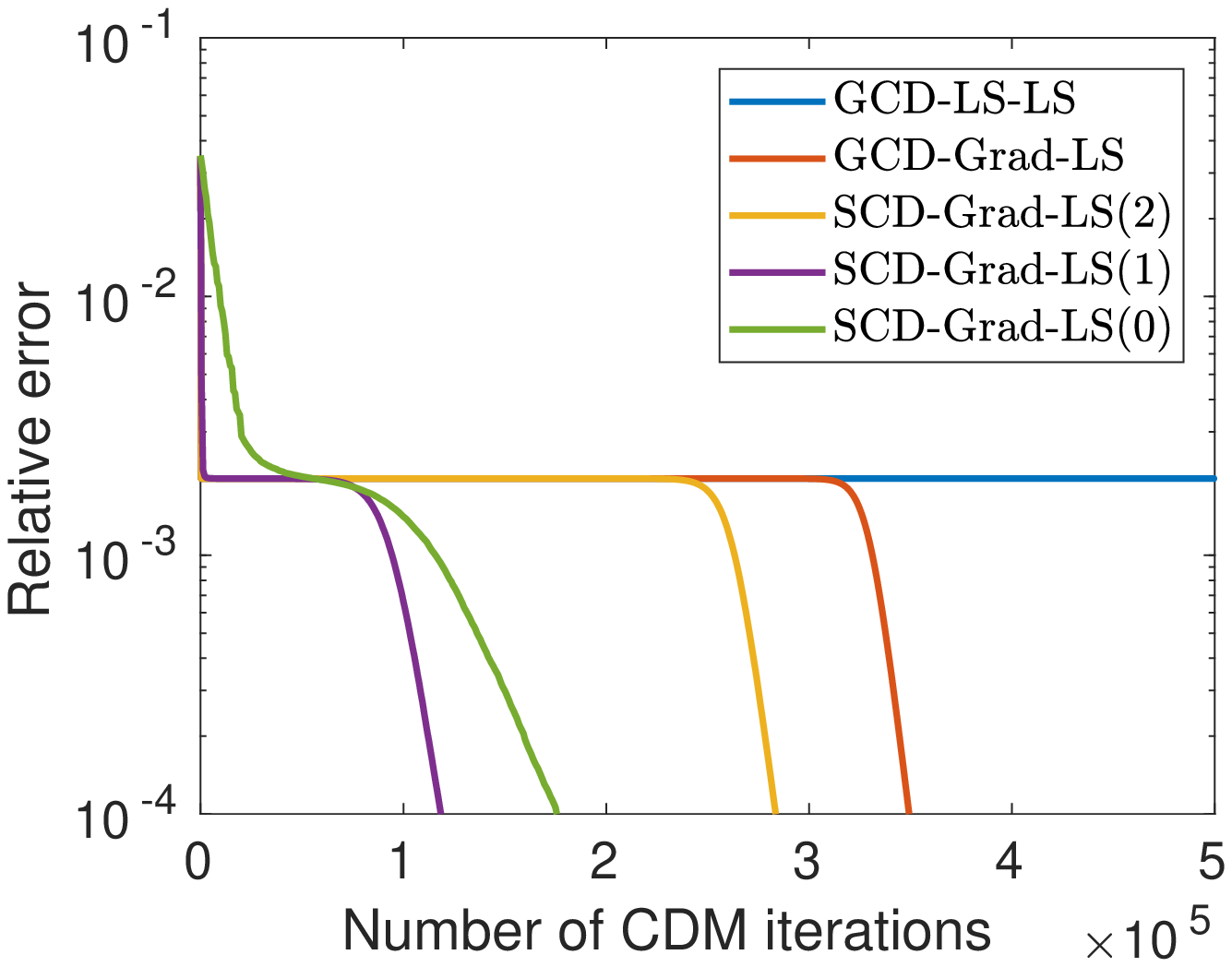}
        \caption{The convergence behavior of greedy CDMs and stochastic
        CDMs on $A_{108} - 100I$. The initial vector has $100$ non-zero
        entries on random coordinates.}
    \end{subfigure}
    \caption{Convergence behavior of greedy coordinate-wise methods vs
      stochastic coordinate-wise methods. Both (a) and (b) demonstrate
      the convergence behavior for solving \eqref{eq:LEVP-opt} and the
      matrix $A_{108}$ is a random matrix of size 5000 with largest
      eigenvalue being 108 and other eigenvalue equally distributed in
      $[1,100)$. Note that we shift the matrix in (b) as
      $A_{108} - 100 I$ so that the only saddle point is the origin.}

    \label{fig:gcdm-vs-scdm}
\end{figure}

\subsection{SCD-Grad-vecLS and SCD-Grad-LS}

\begin{algorithm}[htbp]
    \caption{SCD-Grad-vecLS for LEVP}
    \label{alg:SCD-Grad-vecLS}

    Input: Symmetric matrix $A \in \bbR^{n \times n}$; initial vector
    $\vec{x}{0}{}$; probability power $t$; number of coordinates $k$.

\begin{algorithmic}[1]

    \STATE $\vec{z}{0}{} = A\vec{x}{0}{}$

    \STATE $\ell = 0$

    \WHILE {(not converged)} 

        \STATE $\nu = \norm{\vec{x}{\ell}{}}^2$

        \FOR {$j = 1, 2, \dots, n$}

            \STATE $\vec{c}{\ell}{j} = \nu \vec{x}{\ell}{j} -
            \vec{z}{\ell}{j}$

        \ENDFOR

        \STATE Sample $k$ coordinates with probability proportional
        to $\abs{\vec{c}{\ell}{}}^t$, denote it as $\Omega$

        \STATE Solve \eqref{eq:cubicvec} with the \emph{root picking
        strategy} for $\alpha_\ell$

        \STATE $\vec{x}{\ell+1}{j} = 
        \begin{cases}
            \vec{x}{\ell}{j} + \alpha_\ell \grad_j \func{f}{
            \vec{x}{\ell}{} }, & j \in \Omega\\
            \vec{x}{\ell}{j}, & j \not\in \Omega
        \end{cases}$

        \STATE $\vec{z}{\ell+1}{} = \vec{z}{\ell}{} + A_{:,\Omega}\bigl(
        \vec{x}{\ell+1}{\Omega} - \vec{x}{\ell}{\Omega} \bigr)$

        \STATE $\ell = \ell + 1$

    \ENDWHILE
\end{algorithmic}

\end{algorithm}

The first stochastic coordinate-wise descent method we consider is
SCD-Grad-vecLS as in Algorithm~\ref{alg:SCD-Grad-vecLS}. SCD-Grad-vecLS
is a stochastic version of GCD-Grad-LS: Instead of picking
the coordinate with largest magnitude in $\vec{c}{\ell}{j}$,
we sample $k$ coordinates with probability proportional to
$\bigl\lvert\vec{c}{\ell}{j}\bigr\rvert^t$ for $t \geq 0$ and the set
of sampled indices is denoted as $\Omega$, where as in GCD-Grad-LS in
Section~\ref{sec:GCD-Grad-LS}, $\vec{c}{\ell}{}$ is proportional to the
gradient vector, $\grad f(\vec{x}{\ell}{})$. Hence the sampling strategy
here is equivalent to sampling with probability proportional to the
$t$-th power of the absolute value of the gradient vector. Through
the similar derivation as in Section~\ref{sec:GCDM}, we can show
that the exact line search can also be conducted along a given search
direction. We denote the line search objective function as $h(\alpha)
= \func{f}{ \vec{x}{\ell}{} + \alpha v } = \func{f}{ \vec{x}{\ell}{}
+ \alpha \sum_{j \in \Omega} \grad_j \func{f}{\vec{x}{\ell}{}} e_j }$,
where $v = \sum_{j \in \Omega} \grad_j \func{f}{\vec{x}{\ell}{}} e_j$
defines the search direction. Function $h(\alpha)$ is again a quartic
polynomial of $\alpha$. All candidates of optimal $\alpha$ are roots
of the cubic polynomial,
\begin{equation} \label{eq:cubicvec}
    \begin{split}
        \frac{\diff h}{\diff \alpha}(\alpha) = & 
        4 \norm{v_\Omega}^4 \alpha^3 + 12 \left( v_\Omega^\top
        \vec{x}{\ell}{\Omega} \right) \norm{v_\Omega}^2 \alpha^2 \\
        & + 4 \left( \nu \norm{v_\Omega}^2 + 2
        \left( v_\Omega^\top \vec{x}{\ell}{\Omega} \right)^2 -
        v_\Omega^\top
        A_{\Omega,\Omega} v_\Omega \right) \alpha + 4 \nu
        \left( v_\Omega^\top \vec{x}{\ell}{\Omega} \right) - 4
        v_\Omega^\top
        \vec{z}{\ell}{\Omega} = 0,\\
    \end{split}
\end{equation}
where $\nu = \norm{ \vec{x}{\ell}{} }^2 $. We notice that all
coefficients in \eqref{eq:cubicvec}, given $\nu$, can be computed
in $\bigO(k^2)$ operations for $k$ being the number of indices in
$\Omega$. Hence all candidates of the optimal $\alpha$ can be obtained
in $\bigO(k^2)$ operations. Analog to Section~\ref{sec:GCDM}, we adopt
the \emph{root picking strategy} to find the optimal $\alpha$ which
has lowest function value. We conclude that the optimal $\alpha$
is achievable in $\bigO(k^2)$ operations given pre-calculated
$\nu$. Algorithm~\ref{alg:SCD-Grad-vecLS} describes the steps in detail.

\begin{algorithm}[htb]
    \caption{SCD-Grad-LS for LEVP}
    \label{alg:SCD-Grad-LS}

    Input: Symmetric matrix $A \in \bbR^{n \times n}$; initial vector
    $\vec{x}{0}{}$; probability power $t$; number of coordinates $k$.

\begin{algorithmic}[1]

    \STATE $\vec{z}{0}{} = A\vec{x}{0}{}$
    \STATE $\ell = 0$

    \WHILE {(not converged)}
        \STATE $\nu = \norm{\vec{x}{\ell}{}}^2$
        \FOR {$j = 1, 2, \dots, n$}
            \STATE $\vec{c}{\ell}{j} = \nu \vec{x}{\ell}{j} -
            \vec{z}{\ell}{j}$
        \ENDFOR

        \STATE Sample $k$ coordinates with probability proportional
        to $\bigl\vert \vec{c}{\ell}{} \bigr\vert^t$; denote $\Omega$
        the sampled index set

        \FOR {$j \in \Omega$}
            \STATE $p_j = \nu - \bigl( \vec{x}{\ell}{j} \bigr)^2 -
            A_{j,j}$
            \STATE $q_j = A_{j,j} \vec{x}{\ell}{j} - \vec{z}{\ell}{j}$
            \STATE Solve $\alpha_j^3 + p_j \alpha_j + q_j = 0$ for real
            $\alpha_j$
        \ENDFOR
        \STATE $\vec{x}{\ell+1}{j} = \begin{cases}
            \vec{x}{\ell}{j} + \alpha_j, & j \in \Omega\\
                \vec{x}{\ell}{j}, & j \not\in \Omega
              \end{cases}$
        \STATE $\vec{z}{\ell+1}{} = \vec{z}{\ell}{} + A_{:,\Omega}\bigl(
        \vec{x}{\ell+1}{\Omega} - \vec{x}{\ell}{\Omega} \bigr)$

        \STATE $\ell = \ell + 1$
    \ENDWHILE
\end{algorithmic}

\end{algorithm}

One drawback of Algorithm~\ref{alg:SCD-Grad-vecLS} is that the exact
line search relies on all selected coordinates, which is unsuitable for
asynchronous implementation. Here we propose an aggressive SCD method,
SCD-Grad-LS, which can be implemented in an asynchronized fashion. The
method combines the coordinate picking strategy in SCD-Grad-vecLS and
coordinate updating strategy in GCD-Grad-LS. But the updating strategy in
SCD-Grad-LS updates each coordinate with the coordinate-wise exact line
search independently and can be delayed. Algorithm~\ref{alg:SCD-Grad-LS}
is the pseudo code of SCD-Grad-LS.

Updating $k > 1$ coordinates independently as in
Algorithm~\ref{alg:SCD-Grad-LS} does not have guarantee of
convergence. When $k$ is large or $t$ is large such that the same
coordinate is updated multiple times, we do observe non-convergent
behavior of the iteration in practice. However, as shown in the numerical
results, when $k$ is relatively small compared to $n$ and $t=1,2$,
SCD-Grad-LS, on average, requires about a fraction of $1/k$ number of
iterations. There exists a fix to guarantee the convergence for any
$k>1$. Instead of updating as $\vec{x}{\ell+1}{} = \vec{x}{\ell}{} +
\sum_{j \in \Omega} \alpha_j e_j$, we update as
\begin{equation}\label{eq:modupdaterule}
\vec{x}{\ell+1}{} = \vec{x}{\ell}{} + \frac{1}{k}\sum_{j \in \Omega}
\alpha_j e_j,
\end{equation}
where $\alpha_j$ is the optimal step size in $j$-th coordinate as in
\eqref{eq:cubicopt}.  Such a change enables convergence but usually
increases the iteration number by a factor of $k$ if both cases
converge.

In the sampling procedure of stochastic CDMs, there are two ways of
sampling, sampling with replacement and sampling without replacement.  We
claim that when $k \ll n$ and the variance of $\abs{\vec{c}{\ell}{}}^t$
is small, sampling with or without replacement behaves very similarly.
While, for $\abs{\vec{c}{\ell}{}}^t$ with large variance or $k \approx
n$, these two sampling strategies behave drastically differently. For
example, when $t \rightarrow \infty$, probability proportional to
$\abs{\vec{c}{\ell}{}}^t$ becomes an indicator vector on a single
coordinate (assuming non-degeneracy). Sampling with replacement results
a set $\Omega$ of $k$ same indices, whereas sampling without replacement
results a set of $k$ different indices corresponding to the largest $k$
entries in $\abs{\vec{c}{\ell}{}}$. In the analysis below, we prove the
local convergence analysis of the Algorithm~\ref{alg:SCD-Grad-vecLS}
and Algorithm~\ref{alg:SCD-Grad-LS} when $k=1$ (the two algorithms are
equivalent when $k=1$). Similar but more complicated analysis could
be done for $k\geq 1$ when sampling with replacement is adopted and
the modified updating strategy \eqref{eq:modupdaterule} is adopted in
Algorithm~\ref{alg:SCD-Grad-LS}.

\subsection{Local convergence of stochastic coordinate-wise descent
method}
\label{sec:local-conv}

In this section, we analyze the convergence properties of stochastic
coordinate-wise descent methods for SCD-Grad-LS. We will present the
analysis for a general strongly convex objective function with Lipschitz
continuous gradient. The analysis uses the following notations and
definitions. We take
\begin{equation}
  B^\pm =
  \Set{ y | \norm{ y \mp \sqrt{\lambda_1}v_1} \leq \frac{1}{30}
    \frac{\min(2\lambda_1,\lambda_1-\lambda_2)}{\sqrt{\lambda_1}} }
\end{equation}
as two neighborhoods around global minimizers
$\pm \sqrt{\lambda_1} v_1$ respectively.

\begin{definition}[Coordinate-wise Lipschitz continuous]
\label{def:Lipschitz}
    Let a function $g(x): \bbS \mapsto \bbR$ be continuously
    differentiable. The gradient of the function $\grad g(x)$ is
    coordinate-wise Lipschitz continuous on $\bbS$ if there exists
    a positive constant $L$ such that,
    \begin{equation}
        \abs{\grad_i g(x+\alpha e_i) - \grad_i g(x)} \leq L \abs{\alpha},
        \quad \forall\, x, x + \alpha e_i \in \bbS \text{ and }
        i=1,2,\dots, n.
    \end{equation}
\end{definition}

Compared to the usual Lipschitz constant of $\grad g(x)$, denoted as
$L_g$, we have the relation, $L \leq L_g$. If $g(x)$ is further assumed
to be convex, then we have $L_g \leq nL$. If $g(x)$ is
twice-differentiable, Definition~\ref{def:Lipschitz} is equivalent to
$\abs{e_i^\top \grad^2 g(x) e_i} \leq L$ for all $i = 1, 2, \dots, n$ and
$x \in \bbS$. An important consequence of a coordinate-wise Lipschitz
continuous function $g(x)$ is that
\begin{equation} \label{eq:Lip-bound}
    g(x+\alpha e_i) \leq g(x) + \grad_{i} g(x) \alpha +
    \frac{L}{2}\alpha^2, \quad \forall\, x \in \bbS \text{ and } x +
    \alpha e_i \in \bbS.
\end{equation}

The following lemma extends \cite[Lemma A.3]{Lei2016} to objective
function $f(x)$ with matrix $A$ satisfying Assumption~\ref{ass:matA}.
\begin{lemma} \label{lem:Lip-bound}
    Function $f(x)$ defined in \eqref{eq:LEVP-opt} is a continuously
    differentiable function. Further the gradient function $\grad
    f(x)$ is coordinate-wise Lipschitz continuous on either $B^+$
    or $B^-$ with constant $L = 12\lambda_1 + 2 \min(2\lambda_1,
    \lambda_1-\lambda_2) + 4\max_i
    \abs{A_{i,i}}$.
\end{lemma}

\begin{definition}[\strongcvxity{p}] \label{def:strongly-convex}
    Let a function $g(x): \bbS \mapsto \bbR$ be continuously
    differentiable. It is said to be \strongcvx{p}, if there exists a
    constant $\mu_p > 0$ such that
    \begin{equation} \label{eq:strongly-convex}
        g(y) \geq g(x) + \grad g(x)^\top (y-x) + \frac{\mu_p}{2}
        \norm{y-x}_{p}^2, \quad \forall x, y \in \bbS.
    \end{equation} 
\end{definition}
Throughout the paper, we assume $p \geq 1$ in the definition.
Definition~\ref{def:strongly-convex} is a generalized version of
the traditional strong convexity, which corresponds to $p=2$ case.

Combine \eqref{eq:strongly-convex} with the equivalence of different
norms in finite dimensional vector space, we obtain for any $p \geq q
\geq 1$
\begin{equation}
    \begin{split}
        g(y) & \geq  g(x) + \grad g(x)^\top (y-x) + \frac{\mu_p}{2}
        \norm{y-x}_{p}^2 \\ 
        & \geq  g(x) + \grad g(x)^\top (y-x) + n^{\nicefrac{2}{p}
        - \nicefrac{2}{q}}\frac{\mu_p}{2} \norm{y-x}_{q}^2, \quad
        \forall x, y \in \bbS.
\end{split}
\end{equation}
Therefore, if $g(x)$ is a \strongcvx{p} function with constant $\mu_p$,
then $g(x)$ is a \strongcvx{q} function with constant $\mu_{q}
\geq n^{\nicefrac{2}{p} - \nicefrac{2}{q}}\mu_p$,
which is equivalent to $n^{\nicefrac{2}{q}}\mu_{q} \geq
n^{\nicefrac{2}{p}}\mu_p$.  On the other hand side, using the
inequality of vector norm together with \eqref{eq:strongly-convex},
we obtain for any $p \geq q \geq 1$,
\begin{equation}
    \begin{split}
        g(y) &\geq  g(x) + \grad g(x)^\top (y-x) + \frac{\mu_{q}}{2}
        \norm{y-x}_{q}^2 \\
        & \geq  g(x) + \grad g(x)^\top (y-x) + \frac{\mu_{q}}{2}
        \norm{y-x}_{p}^2, \quad \forall x, y \in \bbS.
\end{split}
\end{equation}
Therefore, if $g(x)$ is a \strongcvx{q} function with constant
$\mu_{q}$, then $g(x)$ is a \strongcvx{p} function with constant
$\mu_{p} \geq \mu_{q}$. Putting two parts together, we have the
following inequalities of $\mu_p$ and $\mu_q$
\begin{equation} \label{eq:mu-ineq}
    n^{\nicefrac{2}{p} - \nicefrac{2}{q}} \mu_p \leq \mu_q
    \leq \mu_p,
\end{equation}
where $p \geq q \geq 1$.

\begin{lemma} \label{lem:smooth}
    Function $f(x)$ defined in \eqref{eq:LEVP-opt} is
    \strongcvx{2} on either $B^+$ or $B^-$ with constant $\mu_2 =
    3 \min(2\lambda_1,\lambda_1 - \lambda_2)$, and hence, there exists
    $\mu_p$ such that $f(x)$ is \strongcvx{p} for any $p \geq 1$.
\end{lemma}

The proof of the \strongcvxity{2} in Lemma~\ref{lem:smooth} follows
an extension of the proof of Lemma A.2 in \cite{Lei2016}, where the
minimum eigenvalue of the Hessian matrix is modified according to
the assumption of the matrix $A$. Combining with \eqref{eq:mu-ineq},
we have the existence of $\mu_p$ for all $p \geq 1$.

$B^+$ and $B^-$ are two disjoint 2-norm ball around global
minimizers. Next we define two sublevel sets $D^+$ and $D^-$,
contained in $B^+$ and $B^-$ respectively as 
\begin{equation}\label{eq:defD}
    D^\pm = \Set{ x \in B^\pm | f(x) \leq \min_{y \in \partial B^\pm}
    f(y) },
\end{equation}
where $\partial B^\pm$ denote the boundary of $B^\pm$.  Obviously,
two global minimizers lie in $D^\pm$ respectively, i.e., $\pm
\sqrt{\lambda_1}v_1 \in D^\pm$.  Lemma~\ref{lem:mono-decay}
shows monotonic decay property of the iteration defined by
Algorithm~\ref{alg:SCD-Grad-LS} once the iterations falls in $D^\pm$.
It also shows that $D^+ \cup D^-$ is a contraction set for the
iteration.

\begin{lemma} \label{lem:mono-decay}
    Consider function $f(x)$ as defined in \eqref{eq:LEVP-opt} and the
    iteration follows Algorithm~\ref{alg:SCD-Grad-LS} with $k=1$. For
    any $\vec{x}{\ell}{} \in D^+ \cup D^-$, we have
    \begin{equation} \label{eq:mono-decay}
        \func{f}{\vec{x}{\ell+1}{}} \leq \func{f}{ \vec{x}{\ell}{}
        } - \frac{1}{2L}\left( \grad_{j_\ell} \func{f}{ \vec{x}{\ell}{} }
        \right)^2,
    \end{equation}
    where $j_\ell$ is the index of the coordinate being picked. Moreover,
    we have  $\vec{x}{\ell+1}{} \in D^+ \cup D^-$.
\end{lemma}

The proof of Lemma~\ref{lem:mono-decay} can be found in
Appendix~\ref{app:local-conv}.

In Algorithm~\ref{alg:SCD-Grad-LS}, we notice that the iteration of
$\vec{x}{\ell}{}$ in the SCD-Grad-LS samples coordinate $j$ with
probability proportional to
$\abs{\grad_j f\left( \vec{x}{\ell}{} \right)}^t$ for some
non-negative power $t$.  In the following lemma and theorem, we adopt
notation $f^*$ as the minimum of the function and $X^*$ be the set of
minimizers, i.e., $X^* = \Set{\pm \sqrt{\lambda_1}v_1}$. A distance
function between two sets or between a point and a set is defined as,
$\dist{S_1}{S_2} = \min_{x \in S_1, y \in S_2} \norm{x-y}$.

\begin{lemma} \label{lem:conv}
    Consider function $f(x)$ as defined in \eqref{eq:LEVP-opt} and the
    iteration follows Algorithm~\ref{alg:SCD-Grad-LS} with $k=1$. For
    any $\vec{x}{\ell}{} \in D^+ \cup D^-$,
    \begin{equation} \label{eq:local-conv}
        \expect{ f\bigl( \vec{x}{\ell+1}{} \bigr) \mid \vec{x}{\ell}{}
        } - f^* \leq \left( 1-\frac{\mu_q}{L n^{2-\frac{2}{q}}}
        \right) \left( f\bigl( \vec{x}{\ell}{} \bigr) - f^*\right),
    \end{equation}
    where $q = \frac{t+2}{t+1}$.
\end{lemma}

The proof of Lemma~\ref{lem:conv} can be found in
Appendix~\ref{app:local-conv}.

\begin{theorem} \label{thm:scd-conv}
    Consider function $f(x)$ as defined in \eqref{eq:LEVP-opt} and the
    iteration follows Algorithm~\ref{alg:SCD-Grad-LS} with $k=1$. For
    any $\vec{x}{0}{} \in D^+ \cup D^-$,
    \begin{equation} 
        \expect{ f\bigl( \vec{x}{\ell}{} \bigr) \mid \vec{x}{0}{}
        } - f^* \leq \left( 1-\frac{\mu_q}{L n^{2-\frac{2}{q}}}
        \right)^{\ell} \left( f\bigl( \vec{x}{0}{} \bigr) - f^*\right),
    \end{equation}
    where $q = \frac{t+2}{t+1}$.  Moreover,
    \begin{equation} 
        \expect{ \dist{\vec{x}{\ell}{}}{X^*}^2 \mid \vec{x}{0}{} }
        \leq \frac{2}{\mu_2} \left( 1-\frac{\mu_q}{L n^{2-\frac{2}{q}}}
        \right)^{\ell} \left( f\bigl( \vec{x}{0}{} \bigr) - f^*\right).
    \end{equation}
\end{theorem}

The proof of Theorem~\ref{thm:scd-conv} can be found in the
Appendix~\ref{app:local-conv}. Theorem~\ref{thm:scd-conv}
is slightly more complicated than Lemma~\ref{lem:conv} since
Algorithm~\ref{alg:SCD-Grad-LS} adopts the exact line search and there
are two sublevel sets $D^\pm$ for the non-convex objective function
$f(x)$. The iteration $\vec{x}{\ell}{}$ might jump between the two
sets. The monotonicity of the exact line search is the key to extend
Lemma~\ref{lem:conv} to Theorem~\ref{thm:scd-conv}.

\begin{remark}\label{rmk:scd-conv-mono}
    According to Theorem~\ref{thm:scd-conv}, the convergence
    rate depends on $n^{\nicefrac{2}{q}}\mu_q$, which depends
    on $q$ and hence $t$.  The right side of \eqref{eq:mu-ineq}
    indicates that, for a problem such that $\mu_1 = \mu_q = \mu_2$,
    the convergence rate of $q=1 \Leftrightarrow t=\infty$ is $n$
    times larger than that of $q=2 \Leftrightarrow t=0$. This means
    that greedy CDM could be potentially $n$ times faster than
    stochastic CDM with uniform sampling. According to the left
    side of inequality \eqref{eq:mu-ineq}, for $p \geq q$, we have
    $n^{\nicefrac{2}{p}}\mu_p \leq n^{\nicefrac{2}{q}}\mu_q$, which
    means the convergence rate of $q$ is equal to or faster than that
    of $p$. In terms of $t$, we conclude that the convergence rate
    of smaller $t$ is smaller than that of larger $t$ and larger $t$
    potentially leads to faster convergence.
\end{remark}

\begin{remark}
    For a fair comparison between CDMs with traditional methods
    such as gradient descent, we should take $n$ iterations of
    CDM and compare it against a single iteration of gradient
    descent since one iteration of CDM only updates one coordinate
    whereas one iteration of gradient descent updates $n$
    coordinates. Under such a setting, Theorem~\ref{thm:scd-conv}
    implies that
    \begin{align*} 
        \expect{ f\bigl( \vec{x}{n\ell}{} \bigr) \mid \vec{x}{0}{}
        } - f^* &\leq \left( 1-\frac{\mu_q}{L n^{2-\frac{2}{q}}}
        \right)^{n\ell} \left( f\bigl( \vec{x}{0}{} \bigr)
        - f^*\right) \\ &\leq \bigl(e^{-\frac{\mu_q}{L}
        n^{\frac{2}{q}-1}}\bigr)^\ell \left( f\bigl( \vec{x}{0}{}
        \bigr) - f^*\right), \\
    \end{align*}
    and
    \begin{align*} 
        \expect{ \dist{\vec{x}{n\ell}{}}{X^*}^2 \mid \vec{x}{0}{}
        } \leq \frac{2}{\mu_2} \bigl(e^{-\frac{\mu_q}{L}
        n^{\frac{2}{q}-1}}\bigr)^\ell \left( f\bigl( \vec{x}{0}{}
        \bigr) - f^*\right).
    \end{align*}
    For uniform sampling, $t = 0$ and $q = 2$, the decreasing
    factor is $e^{-\frac{\mu_q}{L} n^{\frac{2}{q}-1}} =
    e^{-\frac{\mu_2}{L}}$, which is independent of $n$. This result
    implies Algorithm~\ref{alg:SCD-Grad-LS} with uniform sampling
    has the same convergence rate as gradient descent. Combining
    with Remark~\ref{rmk:scd-conv-mono}, we conclude that
    Algorithm~\ref{alg:SCD-Grad-LS} with arbitrary $t\geq 0$ converges
    at least as fast as gradient descent.
\end{remark}

\begin{remark}
    If we extend Definition~\ref{def:Lipschitz} to multi-coordinate
    Lipschitz continuity, and assume $f(x)$ is multi-coordinate
    Lipschitz continuous with $\widetilde{L}$, a natural extension of
    Theorem~\ref{thm:scd-conv} would follow.  Unfortunately, if the
    coordinates are sampled independently, in an extreme case when
    $k$ same coordinates are sampled, the convergence rate would be
    the same as that in Theorem~\ref{thm:scd-conv}.  In the end,
    the convergence rate would not be improved and could be even
    worse. This is similar to the argument of CDM vs. full gradient
    descent method. In practice, when $k$ is not too large, and the
    sampled coordinates are distinct, we do observe $k$-fold speed-up
    of the iterations.  This observation suggests that the convergence
    rate in Theorem~\ref{thm:scd-conv} is not sharp for $k>1$.
\end{remark}

\begin{figure}[htp]
    \centering
    \begin{subfigure}[t]{0.48\textwidth}
        \includegraphics[width=\textwidth]{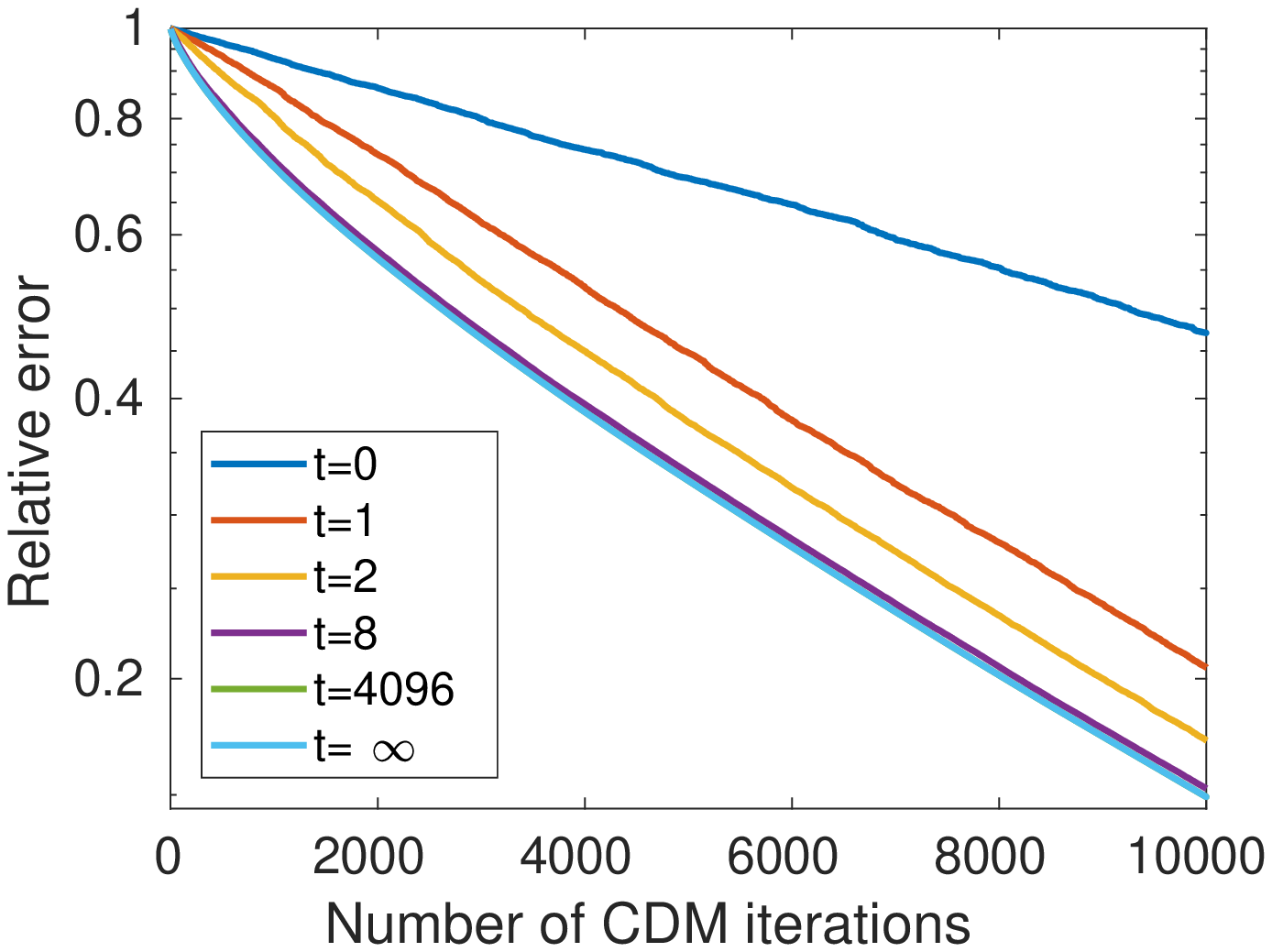}
        \caption{single-coordinate updating}
    \end{subfigure}
    ~
    \begin{subfigure}[t]{0.48\textwidth}
        \includegraphics[width=\textwidth]{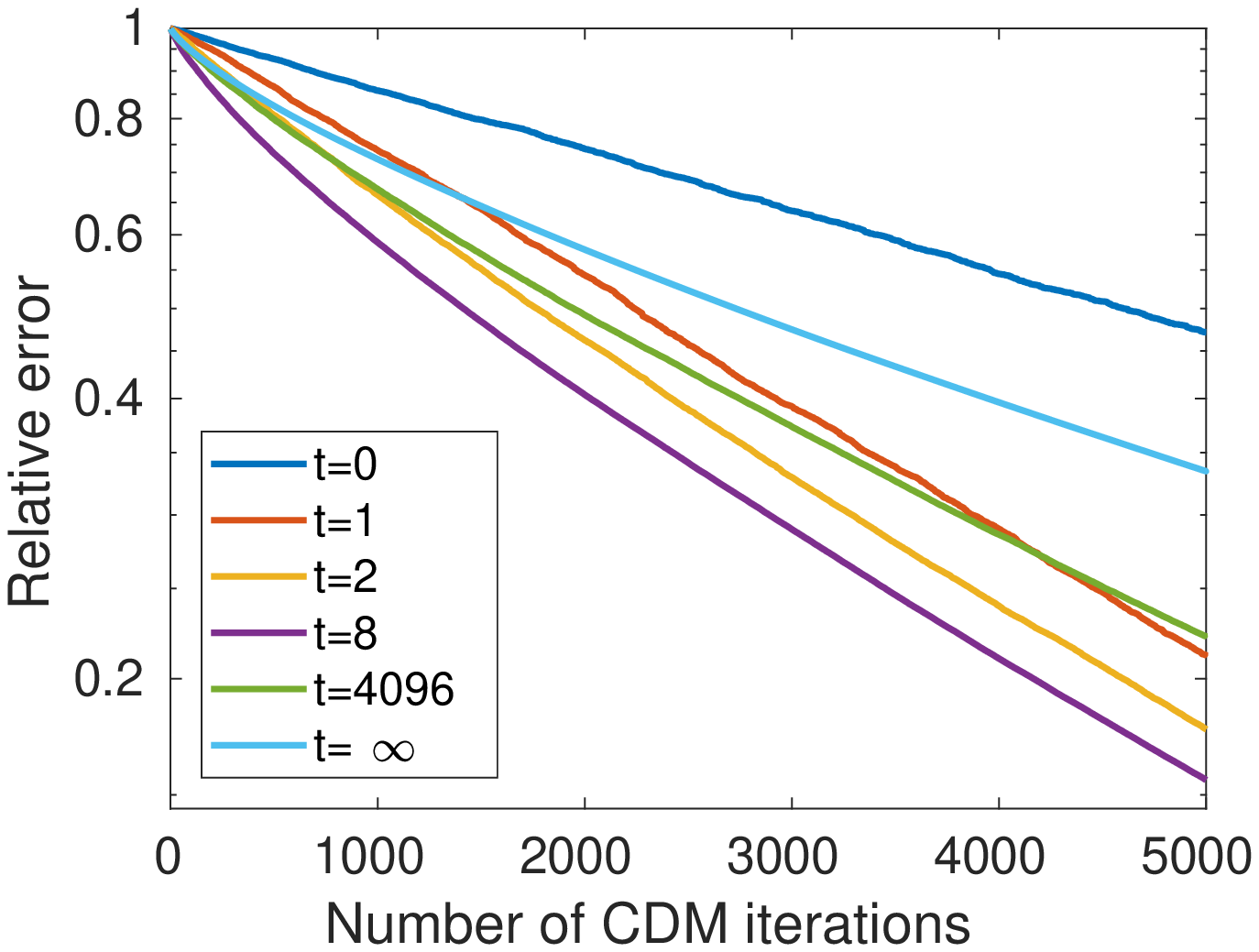}
        \caption{double-coordinate updating}
    \end{subfigure}
    \caption{Convergence behavior of SCD-Grad-Grad method with different
    choice of $t$. The axis of iterations on (a) is twice as long as
    that in (b).}

    \label{fig:local-conv-t}
\end{figure}

Combining Theorem~\ref{thm:scd-conv} together with the monotonicity
of $\mu_q$ as \eqref{eq:mu-ineq} shows that the lower bound of the
convergence rate increases monotonically as $q$ decreases. In terms of
$t$, the larger $t$ corresponds to smaller $q$ which leads to larger
$\mu_q$ for $q = \frac{t+2}{t+1}$. Here we would like to argue that
the equality of \eqref{eq:mu-ineq} is achievable which demonstrates
the power of CDMs with such a sampling strategy. Consider a simple
example $f(x) = \norm{Ax-b}^2$ for $A$ being a diagonal matrix with
diagonal entries $A_{i,i} = 1000$, $i=1,2,\dots,4999$, $A_{5000,5000}
= 1$ and $b = \begin{bmatrix} 0 & \cdots & 0 & 1
\end{bmatrix}^\top$ being a column vector of size $5000$. Through
elementary calculations, we can show that $\grad^2 f(x) = 2A^2$, $\mu_0
= 2$ and $\mu_\infty \approx 1.99$. Therefore all $\mu_t$ lies in the
small interval $[1.99,2]$ and different choice of $t$ would lead to
different rate of convergence. We minimize $f(x)$ using the SCD-Grad-Grad
method. The step size is fixed to be $10^{-6}$ and the initial vector
is chosen randomly. Figure~\ref{fig:local-conv-t} (a) demonstrates the
first $10^4$ iterations of the SCD-Grad-Grad with different choice of
$t$. The relative error is defined to be $\bigl(\func{f}{\vec{x}{\ell}{}}
- \func{f}{x^*}\bigr) / \func{f}{\vec{x}{0}{}}$. Clearly, different
choice of $t$ leads to different convergence rates. And we also notice
that the most significant improvement appears when increasing $t$ from 0
to 1, which implies that sampling with respect to magnitude of gradient
is much better than sampling uniformly. In Figure~\ref{fig:local-conv-t}
(a), the line with $t=4096$ and that with $t=\infty$ overlap.

According to either Theorem~\ref{thm:scd-conv} or
Figure~\ref{fig:local-conv-t} (a), it is convincing that for
single-coordinate updating CDM, picking the index with largest gradient
magnitude should be the optimal strategy. Figure~\ref{fig:local-conv-t}
(b) investigates the case with double-coordinate updating, i.e., sample
two coordinates independently and update two entries simultaneously in
every iteration. Comparing Figure~\ref{fig:local-conv-t} (a) and (b), we
notice for small $t$, double-coordinate updating strategy almost reduces
the required number of iterations by half. As $t \rightarrow \infty$,
the sampling approaches choosing the coordinate with largest gradient
magnitude with probability $1$. Sampling two coordinates independently
thus results in sampling the same index twice and the iteration behaves
similar to the single-coordinate updating CDM.  Readers may argue that
we should sample without replacement or sampling two indices with
two largest gradient magnitudes and updating accordingly.  However,
when the objective function is non-convex, for example, the objective
function $f(x)$ defined in \eqref{eq:LEVP-opt}, the iteration usually
stick in the middle of iteration (see Figure~\ref{fig:gcdm-vs-scdm} (a)).

\begin{remark}[Local convergence of greedy coordinate-wise descent method]
    The local convergence analysis of SCDMs can be extended to
    the greedy coordinate-wise descent methods.  GCD-Grad-LS as in
    Algorithm~\ref{alg:GCD-LS-LS} is a special case of SCD-Grad-LS with
    $t = \infty$ and $k=1$. \footnote{ Previous work~\cite{Lei2016}
    provides a local convergence proof for GCD-Grad-LS, which uses the
    theorem in \cite{Nesterov2012,Nutini2015}. Unfortunately, there is a
    small gap in the proof. The local domain defined in the proof is not
    a contractive domain of the iteration. Instead, the $D^{\pm}$ defined
    in \eqref{eq:defD} is local domain of contraction, and can be used
    to fill in the gap of the proof.} Hence Theorem~\ref{thm:scd-conv}
    proves the local convergence of the GCD-Grad-LS with $q = 1$.

GCD-LS-LS as in Algorithm~\ref{alg:GCD-LS-LS} conducts optimal
coordinate-wise descent every iteration. Therefore, for GCD-LS-LS,
Lemma~\ref{lem:mono-decay} can be proved for any index of coordinate $j$
and then Lemma~\ref{lem:conv} holds for $q = 1$. The modified lemmas
lead to the following corollary.

\begin{corollary} \label{cor:gcd-conv}
    Consider function $f(x)$ as defined in \eqref{eq:LEVP-opt} and
    the iteration follows Algorithm~\ref{alg:GCD-LS-LS}. For any
    $\vec{x}{0}{} \in D^+ \cup D^-$,
    \begin{equation} \label{eq:local-conv-gcd1}
        \expect{ f\bigl( \vec{x}{\ell}{} \bigr) \mid \vec{x}{0}{}
        } - f^* \leq \left( 1-\frac{\mu_1}{L}
        \right)^{\ell} \left( f\bigl( \vec{x}{0}{} \bigr) - f^*\right).
    \end{equation}
    Moreover,
    \begin{equation} \label{eq:local-conv-gcd2}
        \expect{ \dist{\vec{x}{\ell}{}}{X^*}^2 \mid \vec{x}{0}{} }
        \leq \frac{2}{\mu_2} \left( 1-\frac{\mu_1}{L}
        \right)^{\ell} \left( f\bigl( \vec{x}{0}{} \bigr) - f^*\right).
    \end{equation}
\end{corollary}

The local convergence of GCD-Grad-LS and GCD-LS-LS are achieveable,
while the global convergences are so far still open. Numerically,
both methods converge to global minima for most the cases we tested.
\end{remark}

\begin{remark}[Convergence behavior in non-convex area]
    We demonstrate one example such that the stochastic CDMs
    outperforms greedy CDMs even all methods are conducting single
    coordinate updating. We generate a random matrix $A_{108} -
    100 I$ of size 5000, whose the largest eigenvalue is 8 and other
    eigenvalues are distributed equally on $[-99,0)$. According to
    Lemma~\ref{lem:strict-saddle-point}, $0$ is the only strict saddle
    point of the problem. If all methods start with a sparse initial
    vector, in this problem, they all converge to the strict saddle
    point first, and then escape from it. Figure~\ref{fig:gcdm-vs-scdm}
    (b) shows the convergence behavior of five methods ordered by the
    degree of greediness. The horizontal line is the relative error
    at the strict saddle point 0. As we can see that it takes more
    time for greedier method to escape from the saddle point. And
    although SCD-Grad-LS(0) escape from the saddle point fastest,
    but the overall performance is worse than SCD-Grad-LS(1). This
    motivates us to apply SCD-Grad-LS with small $t$ in practice. Such
    methods are almost as efficient as the greedy methods, as shown in
    Figure~\ref{fig:local-conv-t}, they are also robust to non-convex
    objective functions near saddle points.
\end{remark}

\section{Numerical results}
\label{sec:num-res}

In this section, we perform numerical tests of the methods mentioned in
previous sections on two different kinds of matrices. The first numerical
test calculates the largest eigenvalue of a random symmetric dense
matrix, where we have control of the difficulty of the problems. The
leading eigenvector is evenly distributed among all the entries. The
second test calculates the ground energy and ground state of the two
dimensional (2D) Hubbard model, which is one of our target applications
for the quantum many-body calculation. The ground state problem in
quantum many-body system is equivalent to find the smallest eigenvalue
of a huge sparse symmetric matrix, with each coordinate representing
the coefficient of a Slater determinant. Although the eigenvector in
this example is dense, it is also `sparse' in some sense, which means
a great majority of the entries is close to zero.  Without eigenvector
compression, for the quantum many-body calculation, the state-of-the-art
algorithms are in fact variants of power iterations. Hence in this
paper, we compare coordinate-wise methods against power method as it
is a simple baseline algorithm. Due to the nature of the problem, it
is anticipated that CDMs will have much better performance over the
traditional power method.

Three quantities are used to measure the convergence of the methods:
the square root of the relative error of the objective function
value~\eqref{eq:LEVP-opt}, the relative error of the eigenvalue and the
tangent of the angle between $\vec{x}{\ell}{}$ and the eigenvector. The
square root of the relative error of the objective function value is
defined as
\begin{equation} \label{eq:error_obj}
    \epsilon_{\mathrm{obj}} = \sqrt{ \frac{\func{f}{x^{(\ell)}} -
    f^*}{f^*} },
\end{equation}
where $f^*$ denotes the minimum of $f(x)$. Since $f^*$ is on the scale
of $\lambda_1^2$, we use the square root of the relative error of
the objective function which is on the same scale as $\lambda_1$. To
estimate the eigenvalue, we use
\begin{equation} \label{eq:proj_energy}
E = \frac{x_*^\top Ax^{(\ell)}}{x_*^\top x^{(\ell)}},
\end{equation}
where $x_*$ is the reference vector that overlaps with the eigenvector
$v_1$. In our test, $x_*$ is a unit vector and will be specified
later. This estimator is referred as the projected energy in the
context of quantum chemistry. Thus the relative error is define by
\begin{equation} \label{eq:error_energy}
\epsilon_{\mathrm{energy}} = \frac{\abs{E -
\lambda_1}}{\abs{\lambda_1}}.
\end{equation}
Last,
\begin{equation} \label{eq:error_tan}
\epsilon_{\mathrm{tan}} = \tan{\theta(v^{(\ell)}, v_1)}
\end{equation}
is used to measure the convergence of the eigenvector.

The number of matrix column evaluation is used to measure the efficiency
of the methods. There are several reasons why the number of matrix column
evaluation is a good measure. First it is independent of the computation
environment. Second we assume the evaluation of matrix column is the
most expensive step. It is usually true in quantum many-body calculation
in chemistry (full configuration interaction method) because the matrix
is too large to be stored, the evaluation has to be on-th-fly and the
evaluation of each entry is relatively expensive. Third, it can be used
to compare with other related methods.

All methods are implemented and tested in MATLAB R2017b. The exact
eigenvalue and eigenvector of each problem are calculated by the
``\emph{eigs}'' function in Matlab with high precision.

\subsection{Dense random matrices}

\begin{table}[t] \label{tab:A108}
    \centering
    \caption{Performance of various algorithms for $A_{108}$}
    \small
    \begin{tabular}{rrrrrr}
        \toprule
        Method & $k$ & Min Iter Num & Med Iter Num & Max Iter Num &
        Total Col Access\\
        \toprule
                     PM & 5000 &      - &    135 &      - & 675000 \\
                    CPM &   1 &      - & 309085 &      - & 309085 \\
                    CPM &   4 &      - &  75829 &      - & 303316 \\
                    CPM &  16 &      - &  18809 &      - & 300944 \\
        \midrule
 SI-GSL($\lambda_1$) &   1 &      - & 260000 &      - & 260000 \\
 SI-GSL($1.5\lambda_1$) &   1 &      - & 1002283 &      - & 1002283 \\
        \midrule
             Grad-vecLS & 5000 &      - &     80 &      - & 400000 \\
            GCD-Grad-LS &   1 &      - & 109751 &      - & 109751 \\
              GCD-LS-LS &   1 &      - & 100464 &      - & 100464 \\
        \midrule
      SCD-Grad-vecLS(0) &   4 &  86054 &  94837 & 128467 & 379348 \\
      SCD-Grad-vecLS(0) &  16 &  21409 &  24066 &  30850 & 385056 \\
      SCD-Grad-vecLS(1) &   4 &  36081 &  40866 &  52279 & 163464 \\
      SCD-Grad-vecLS(1) &  16 &   9225 &  10205 &  11934 & 163280 \\
      SCD-Grad-vecLS(2) &   4 &  29867 &  33538 &  45480 & 134152 \\
      SCD-Grad-vecLS(2) &  16 &   7417 &   8334 &  11410 & 133344 \\
         SCD-Grad-LS(0) &   1 & 342945 & 377783 & 483842 & 377783 \\
         SCD-Grad-LS(0) &   4 &  86241 &  96074 & 150090 & 384296 \\
         SCD-Grad-LS(0) &  16 &  21952 &  24487 &  31313 & 391792 \\
         SCD-Grad-LS(1) &   1 & 146161 & 166415 & 223471 & 166415 \\
         SCD-Grad-LS(1) &   4 &  38034 &  41773 &  53257 & 167092 \\
         SCD-Grad-LS(1) &  16 &   9207 &  10479 &  15063 & 167664 \\
         SCD-Grad-LS(2) &   1 & 120411 & 136468 & 177161 & 136468 \\
         SCD-Grad-LS(2) &   4 &  30200 &  34091 &  44993 & 136364 \\
         SCD-Grad-LS(2) &  16 &   7583 &   8631 &  11415 & 138096 \\
        \bottomrule
    \end{tabular}
\end{table}

We first show the advantage of the coordinate-wise descent methods
over power method on dense random matrices. All matrices tested in
this section involve symmetric matrices of size 5000 with random
eigenvectors, i.e.,
\begin{equation}
    A_{\lambda_1} = Q
    \begin{bmatrix}
        \lambda_1 & & & \\
                  & \lambda_2 & & \\
                  & & \ddots & \\
                  & & & \lambda_{5000}
    \end{bmatrix} Q^\top,
\end{equation}
where $\lambda_2, \dots, \lambda_{5000}$ are equally distributed
on $[1,100)$ and $Q$ is a random unitary matrix generated by a QR
factorization of a random matrix with each entry being Gaussian
random number. We use $\lambda_1$ to control the difficulty of the
problem. In particular, we tested three matrices $A_{108}$, $A_{101}$,
and $A_{108}+1000I$.  The last one is a shifted version of $A_{108}$.
It is obvious that $A_{108}+1000I$ is more difficult than $A_{108}$
for power methods. For the optimization problem~\eqref{eq:LEVP-opt},
it is also more difficult since the landscape of the objective function
becomes steeper. While, in practice, CDMs for the optimization problem is
less sensitive to the shift. In this section, all results are reported
in terms of number of column accesses, converging to $10^{-6}$ under
the measure of $\epsilon_{\mathrm{obj}}$. For non-stochastic methods,
we report the number of column accesses from a single run and reported
in the column named ``Med Iter Num''. While, for stochastic methods,
we report the minimum, medium, and maximum numbers of iterations from
100 runs. In all tables, the column named ``Total Col Access'' provides
the total number of matrix column evaluation, which is calculated as
the product of number of coordinates updating each iteration, i.e.,
$k$, and the medium number of iterations.  The initial vector of CDM
is always chosen to be the unit vector $e_1$, with $1$ in the first
coordinate and $0$ in all others.

\begin{table}[t] \label{tab:A101}
    \centering
    \caption{Performance of various algorithms for $A_{101}$}
    \small
    \begin{tabular}{rrrrrr}
        \toprule
        Method & $k$ & Min Iter Num & Med Iter Num & Max Iter Num &
        Total Col Access\\
        \toprule
                     PM & 5000 &      - &    839 &      - & 4195000 \\
                    CPM &   1 &      - & 1899296 &      - & 1899296 \\
                    CPM &   4 &      - & 471779 &      - & 1887116 \\
                    CPM &  16 &      - & 125525 &      - & 2008400 \\
        \midrule
 SI-GSL($\lambda_1$) &   1 &      - & 460000 &      - & 460000 \\
 SI-GSL($1.5\lambda_1$) &   1 &      - & 3972548 &      - & 3972548 \\
        \midrule
             Grad-vecLS & 5000 &      - &    480 &      - & 2400000 \\
            GCD-Grad-LS &   1 &      - & 726093 &      - & 726093 \\
              GCD-LS-LS &   1 &      - & 554521 &      - & 554521 \\
        \midrule
      SCD-Grad-vecLS(0) &   4 & 444954 & 557394 & 788291 & 2229576 \\
      SCD-Grad-vecLS(0) &  16 & 118259 & 140177 & 178029 & 2242832 \\
      SCD-Grad-vecLS(1) &   4 & 207804 & 253035 & 389396 & 1012140 \\
      SCD-Grad-vecLS(1) &  16 &  53497 &  62959 &  84697 & 1007344 \\
      SCD-Grad-vecLS(2) &   4 & 173455 & 206857 & 367199 & 827428 \\
      SCD-Grad-vecLS(2) &  16 &  43618 &  51701 &  82524 & 827216 \\
         SCD-Grad-LS(0) &   1 & 1896835 & 2208812 & 3279223 & 2208812 \\
         SCD-Grad-LS(0) &   4 & 483657 & 571571 & 832675 & 2286284 \\
         SCD-Grad-LS(0) &  16 & 119192 & 139788 & 203427 & 2236608 \\
         SCD-Grad-LS(1) &   1 & 871362 & 1022169 & 1481621 & 1022169 \\
         SCD-Grad-LS(1) &   4 & 206500 & 262616 & 402321 & 1050464 \\
         SCD-Grad-LS(1) &  16 &  54106 &  63607 &  94236 & 1017712 \\
         SCD-Grad-LS(2) &   1 & 717007 & 820728 & 1205225 & 820728 \\
         SCD-Grad-LS(2) &   4 & 176141 & 209155 & 288910 & 836620 \\
         SCD-Grad-LS(2) &  16 &  43237 &  51768 &  75115 & 828288 \\
        \bottomrule
    \end{tabular}
\end{table}

\begin{table}[t] \label{tab:A1108}
    \centering
    \caption{Performance of various algorithms for $A_{108}+1000I$}
      \small
    \begin{tabular}{rrrrrr}
        \toprule
        Method & $k$ & Min Iter Num & Med Iter Num & Max Iter Num &
        Total Col Access\\
        \toprule
                     PM & 5000 &      - &   1214 &      - & 6070000 \\
                    CPM &   1 &      - & 3338202 &      - & 3338202 \\
                    CPM &   4 &      - & 853724 &      - & 3414896 \\
                    CPM &  16 &      - & 214645 &      - & 3434320 \\
         \midrule
 SI-GSL($\lambda_1$) &   1 &      - & 140000 &      - & 140000 \\
 SI-GSL($1.5\lambda_1$) &   1 &      - & 1409635 &      - & 1409635 \\
         \midrule
             Grad-vecLS & 5000 &      - &    875 &      - & 4375000 \\
            GCD-Grad-LS &   1 &      - &  92532 &      - &  92532 \\
              GCD-LS-LS &   1 &      - & 102098 &      - & 102098 \\
         \midrule
      SCD-Grad-vecLS(0) &   4 &  78212 &  92682 & 126872 & 370728 \\
      SCD-Grad-vecLS(0) &  16 &  18632 &  21257 &  30375 & 340112 \\
      SCD-Grad-vecLS(1) &   4 &  31107 &  34023 &  48125 & 136092 \\
      SCD-Grad-vecLS(1) &  16 &   7888 &   8570 &  10915 & 137120 \\
      SCD-Grad-vecLS(2) &   4 &  25047 &  27733 &  37547 & 110932 \\
      SCD-Grad-vecLS(2) &  16 &   6111 &   7000 &   9788 & 112000 \\
         SCD-Grad-LS(0) &   1 & 336724 & 404492 & 486490 & 404492 \\
         SCD-Grad-LS(0) &   4 &  88807 & 100736 & 127954 & 402944 \\
         SCD-Grad-LS(0) &  16 &  22687 &  25715 &  34403 & 411440 \\
         SCD-Grad-LS(1) &   1 & 125623 & 139462 & 192680 & 139462 \\
         SCD-Grad-LS(1) &   4 &  31205 &  34731 &  54808 & 138924 \\
         SCD-Grad-LS(1) &  16 &   8081 &   9524 &  13275 & 152384 \\
         SCD-Grad-LS(2) &   1 & 101177 & 110757 & 147548 & 110757 \\
         SCD-Grad-LS(2) &   4 &  25399 &  31354 &  52048 & 125416 \\
         SCD-Grad-LS(2) &  16 & \multicolumn{4}{c}{do not converge} \\
        \bottomrule
    \end{tabular}
\end{table}

Among the methods compared, PM and Grad-vecLS with $k=5000$
update all coordinates per iteration. All other methods are
coordinate-wise methods and all of them show speedup. PM converges
the slowest. Although the iteration number is small, the cost of
each iteration is huge. CPM reduces the cost by half on average,
which is a big improvement. Grad-vecLS with $k=5000$ is equivalent
to gradient descent with exact line search. Since it updates at the
full gradient direction, the cost is also large. SI-GSL($\lambda_1$)
and SI-GSL($1.5\lambda_1$) adopt $\lambda_1$ and $1.5\lambda_1$ as the
estimation of the leading eigenvalue respectively, and other parameters
in SI-GSL are set to be the underlying true values. SI-GSL($\lambda_1$)
in general performs compariable to GCD-Grad-LS and GCD-LS-LS, while
SI-GSL($1.5\lambda_1$) is siganificantly slower than all optimization
based CDMs. The performance of SI-GSL sensitively depends on the
provided a priori estimation of the leading eigenvalue, which makes it
impractical for our target application settings. The greedy methods
GCD-Grad-LS and GCD-LS-LS speed up significantly. They are about
six times faster than PM for $A_{108}$ and $A_{101}$. While, for
$A_{108}+1000I$, the greedy methods is about 70 times faster. The two
greedy methods themselves have similar cost. For the stochastic methods,
we use gradient information to pick the coordinates and use vecLS or LS
to update. Different probability power $t$ and number of coordinates
updated $k$ are tested. With the same parameters, SCD-Grad-vecLS and
SCD-Grad-LS share the similar performance. For $t=0$, which is equivalent
to uniform sampling, the number of column accesses are almost the same as
the steepest gradient descent. It is reasonable since uniform sampling
does not pick the important coordinate and should behave similar with
the full gradient. For $t=1$ and $t=2$, the speedup is obvious. The
performance for different $t$ agrees with Theorem~\ref{thm:scd-conv},
since larger $t$ shows faster convergence. Another thing to notice is
that the cost is not influenced by the different choices of $k$ for CPM
and SCD. Therefore if we do the real asynchronized version, more speedup
can be gained. Moreover, for stochastic CDMs, the variance is not large.

Comparing the performance for $A_{101}$ and $A_{108}$, the numbers of
column accesses for $A_{101}$ are approximately six times as many as
for $A_{108}$ for all methods. This behavior is anticipated. Comparing
$A_{108}$ to $A_{108} + 1000I$, we see that PM, CPM and Grad-vecLS
slow down significantly by the shift, since the condition number
($\frac{\lambda_1}{\lambda_1 - \lambda_2}$ for PM and the condition
number of the Hessian for gradient descent) increases. However, all
SI-GSL, GCD and SCD methods converge at almost the same speed, or even
a little faster.

\subsection{Hubbard models}

In this test, we calculate the ground energy and ground state of the
2D Hubbard model. The fermion Hubbard model is widely used in condensed
matter physics, which models interacting fermion particles in a lattice
with the nearest neighbor hopping and on-site interaction. It is defined
by the Hamiltonian in the second quantization notation:
\begin{equation} \label{eq:hubbard_real}
    \hat{H} = -t\sum_{\langle r,r'\rangle,\sigma}
    \hat{c}_{r,\sigma}^\dagger \hat{c}_{r',\sigma} + U\sum_r
    \hat{n}_{r\uparrow}\hat{n}_{r\downarrow},
\end{equation}
where $t$ is the hopping strength and $U$ is the on-site interaction
strength. $r, r'$ are indexes of sites in the lattice, and $\langle
r,r'\rangle$ means $r$ and $r'$ are the nearest neighbor. The spin
$\sigma$ takes value of $\uparrow$ and $\downarrow$. The operators
$\hat{c}_{r,\sigma}^\dagger$ and $\hat{c}_{r,\sigma}$ are creation and
annihilation operators, which create or destroy a fermion at site $r$
with spin $\sigma$. They satisfy the canonical anticommutation relation:
\begin{equation*} \label{eq:CAR}
    \{\hat{c}_{r,\sigma}, \hat{c}_{r',\sigma'}^\dagger \} =
    \delta_{r,r'}\delta_{\sigma, \sigma'}, \qquad \{\hat{c}_{r,\sigma},
    \hat{c}_{r',\sigma'} \} = 0, \quad \text{and} \quad
    \{\hat{c}_{r,\sigma}^\dagger, \hat{c}_{r',\sigma'}^\dagger \} = 0,
\end{equation*}
where $\{A, B\} = AB + BA$ is the anti-commutator. 
$\hat{n}_{r,\sigma} = \hat{c}_{r,\sigma}^\dagger \hat{c}_{r,\sigma}$
is the number operator.

When the interaction is not strong, we usually work in the momentum
space. Because in the momentum space, the hopping term in the Hamiltonian
is diagonalized and the ground state is close to the Hatree-Fock (HF)
state. Mathematically speaking, the normalized leading eigenvector
$v_1$ is close to the unit vector $e_{\mathrm{HF}}$ whose nonzero
entry indicates the HF state.  The Hamiltonian in momentum space can
be written as
\begin{equation} \label{eq:hubbard_k}
    \hat{H} = t\sum_{k,\sigma} \varepsilon(k)\hat{n}_{k,\sigma} +
    \frac{U}{N^{\text{orb}}} \sum_{k,p,q}\hat{c}_{p-q,\uparrow}^\dagger
    \hat{c}_{k+q,\downarrow}^\dagger
    \hat{c}_{k,\downarrow}\hat{c}_{p,\uparrow},
\end{equation}
where $\varepsilon(k) = -2 (\cos(k_1) + \cos(k_2))$ and
$N^{\text{orb}}$ is the number of orbitals. In momentum space,
$k$ is the wave number and the creation and annihilation
operator can be obtained by the transform $\hat{c}_{k,\sigma}
= \frac{1}{\sqrt{N^{\text{orb}}}} \sum_r e^{ik\cdot r}
\hat{c}_{r,\sigma}$.

To study the performance of methods, we test several 2D Hubbard models
with different sizes. Here we report the results of two Hubbard
models: both of them are on the $4\times 4$ lattice with periodic
boundary condition. The hopping strength is $t=1$ and the interaction
strength is $U=4$. The first example contains $6$ electrons, with $3$
spin up and $3$ spin down, whereas the second example contains $10$
electrons, with $5$ spin up and $5$ spin down. The properties of the
Hamiltonians are summarized in Table~\ref{tb:hub_ham}. The nonzero
off-diagonal entries of $H$ take value $\pm \frac{U}{N^{\text{orb}}}
= \pm 0.25$, and the diagonal entries distribute in $(-20, 30)$ with
bell-like shape. From the table we see that the Hamiltonians are indeed
sparse. The smallest eigenvalue (Eigenvalues Min) is the ground energy
we are to compute, thus the difference between the smallest one and
the second smallest one is the eigengap.  Since the spectra of $H$ of
the two examples live in the interval $(-100, 100)$, the matrix $A =
100I - H$ is used for calculation such that the smallest eigenvalue of
$H$ becomes the largest one of $A$ and $A$ is positive definite. Both
the initial vector and the reference vector of the energy estimator are
chosen to be $10\,e_{\mathrm{HF}}$ for all calculation, this amounts to
an initial guess of eigenvector $e_{\mathrm{HF}}$ with eigenvalue $10^2
= 100$ on the same scale as the diagonal of $A$ (due to the $100 I$
shift). Notice that the reference vector has only one nonzero entry,
so computing the projected energy is cheap.

\begin{table}[ht]
\centering
    \caption{Properties of the Hubbard Hamiltonian\label{tb:hub_ham}}
    \begin{tabular}{rrcccccc}
        \toprule
        \multirow{2}{*}{$H$} & \multirow{2}{*}{Dim} &
        \multicolumn{3}{c}{nnz per col} &
        \multicolumn{3}{c}{Eigenvalues} \\
        & & Min & Med & Max & Min & 2nd Min & Max \\
        \toprule
        6 electrons & $1.96\times 10^4$ &
        $100$ & $102$ & $112$ & $-14.90$ & $-14.55$ & $20.26$ \\
        10 electrons & $1.19\times 10^6$ &
        $196$ & $202$ & $240$ & $-19.58$ & $-17.08$ & $32.73$ \\
        \bottomrule
    \end{tabular}
\end{table}

For the first example, the $4\times 4$ Hubbard model with $6$
electrons, we run the methods until the relative error of the
objective function value $e_{\mathrm{obj}} < 10^{-6}$. The result is
shown in Table~\ref{tb:hubbard446_result}.

\begin{table}[ht]
    \centering
    \caption{Performance of various algorithms for $4\times 4$ Hubbard
    model with $6$ electrons}\label{tb:hubbard446_result}
    \small
    \begin{tabular}{rrrrrr}
        \toprule
        Method & $k$ & Min Iter Num & Med Iter Num & Max Iter Num &
        Total Col Access\\
        \toprule
                     PM & 19600 &      - &   2255 &      - & 44198000 \\
                    CPM &   1 &      - & 456503 &      - & 456503 \\
                    CPM &  16 &      - &  29293 &      - & 468688 \\
        \midrule
             Grad-vecLS & 19600 &      - &    361 &      - & 7075600 \\
            GCD-Grad-LS &   1 &      - &  31997 &      - &  31997 \\
              GCD-LS-LS &   1 &      - &  30996 &      - &  30996 \\
        \midrule
      SCD-Grad-vecLS(1) &   4 &  36966 &  38763 &  40224 & 155052 \\
      SCD-Grad-vecLS(1) &  16 &  11490 &  12115 &  12644 & 193840 \\
      SCD-Grad-vecLS(1) &  32 &   6283 &   6642 &   6991 & 212544 \\
      SCD-Grad-vecLS(2) &   4 &  18934 &  19260 &  19613 &  77040 \\
      SCD-Grad-vecLS(2) &  16 &   6438 &   6547 &   6652 & 104752 \\
      SCD-Grad-vecLS(2) &  32 &   3880 &   3946 &   4002 & 126272 \\
         SCD-Grad-LS(1) &   1 & 117261 & 120613 & 123744 & 120613 \\
         SCD-Grad-LS(1) &   2 &  58536 &  60470 &  62209 & 120940 \\
         SCD-Grad-LS(1) &   4 &  28716 &  30152 &  30918 & 120608 \\
         SCD-Grad-LS(1) &   8 &  13053 &  15352 &  20197 & 122816 \\
         SCD-Grad-LS(2) &   1 &  47603 &  48136 &  48802 &  48136 \\
         SCD-Grad-LS(2) &   2 &  23789 &  24059 &  24399 &  48118 \\
         SCD-Grad-LS(2) &   4 & \multicolumn{4}{c}{do not converge}\\
         SCD-Grad-LS(2) &   8 & \multicolumn{4}{c}{do not converge} \\
      \bottomrule
    \end{tabular}
\end{table}
In Table~\ref{tb:hubbard446_result}, the second column $k$ is the
number of coordinates updated in one iteration. Each stochastic method
runs for $100$ times and we report the number of iterations and matrix
column accesses, similar to the previous tests.

Since the dimension of the Hubbard model is larger and the eigenvector
$v_1$ is sparser than that in the first dense matrix example, the
advantage of the coordinate-wise methods is more significant. They are
$10^2$ to $10^3$ faster than PM. GCD-Grad-LS and GCD-LS-LS are still
the fastest one. Stochastic CDMs also outperform PM and CPM. Comparing
with the greedy CDMs, they are about 3-5 times slower, which is more
than that in the previous tests.

For the second model with $10$ electrons in the $4\times 4$ grid, we
also compare with other methods, such as SCPM (stochastic version of
CPM), SCD-Uni-Grad (which samples coordinate uniformly and updates the
coordinate as the gradient multiplying a fixed stepsize), CD-Cyc-LS
(which picks the coordinate in a cyclic way and updates the coordinate
using line search), and SCD-Uni-LS (i.e., SCD-Grad-LS(0)).  The step
size, if used, is chosen to be $2$. This step size is larger than that
in Theorem~\ref{thm:CD-Cyc-Grad}, but numerically the iteration
converges. All methods besides PM update single coordinate per
iteration and running up to $10^7$ column accesses or desired
accuracy. Stochastic CDMs sample coordinate with respect to the
magnitude of the gradient, i.e., $t=1$. The convergence of the square
root of the relative error of objective function value
$\epsilon_{\mathrm{obj}}$, eigenvalue estimator
$\epsilon_{\mathrm{energy}}$ and eigenvector estimator
$\epsilon_{\mathrm{tan}}$ are plotted in Figure~\ref{fig:hub4410}.

\begin{figure}[htp]
\centering
    \begin{subfigure}[t]{0.8\textwidth}
    \includegraphics[width=\textwidth]{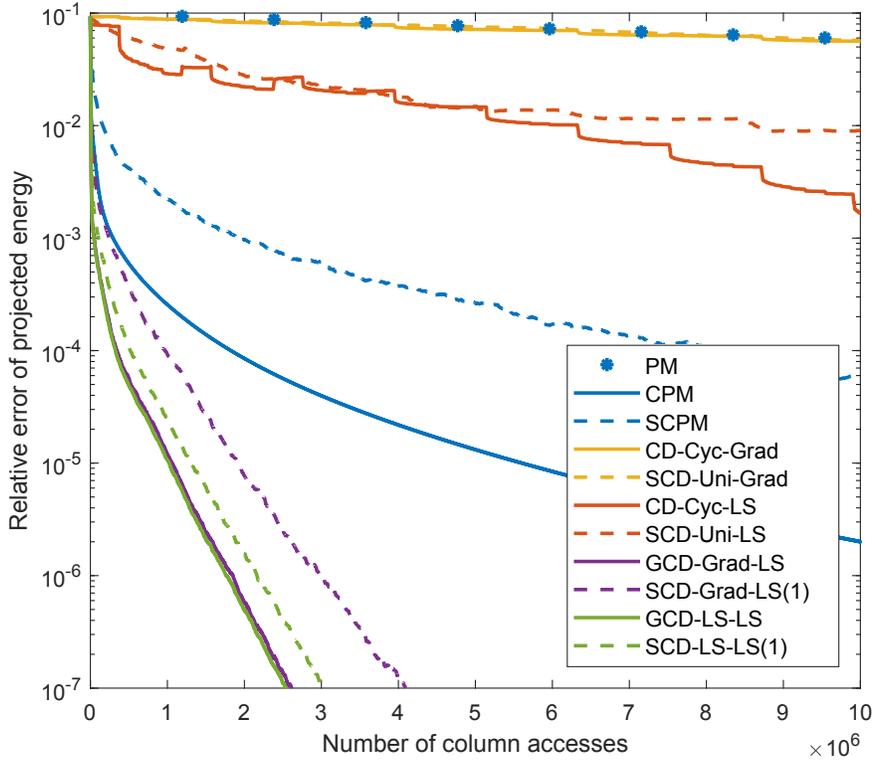}
        \caption{$\epsilon_{\mathrm{energy}}$, Relative error of
        projected energy}
    \end{subfigure}
    \centering
    \begin{subfigure}[t]{0.48\textwidth}
        \includegraphics[width=\textwidth]{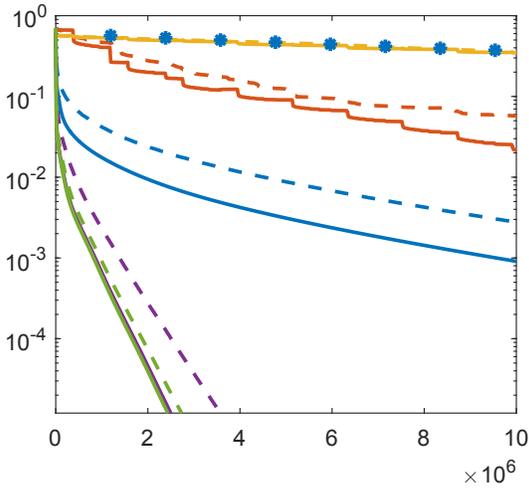}
        \caption{$\epsilon_{\mathrm{tangent}}$, Tangent of angle between
        $v^{(\ell)}$ and $v_1$}
    \end{subfigure}
    ~
    \begin{subfigure}[t]{0.48\textwidth}
        \includegraphics[width=\textwidth]{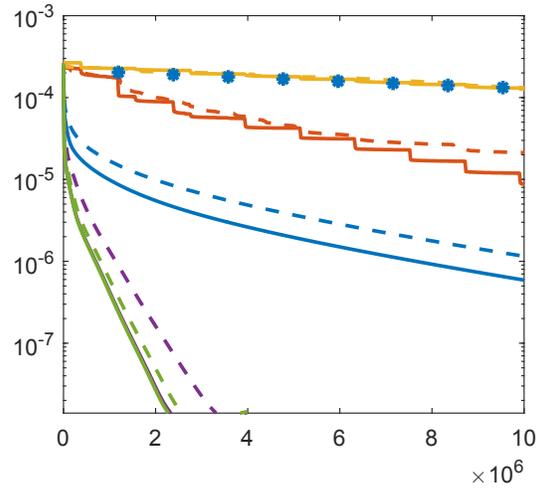}
        \caption{$\epsilon_{\mathrm{obj}}$, Relative error of objective
        function value}
    \end{subfigure}
    \caption{Convergence behavior of the methods for the $4\times 4$
    Hubbard model with $10$ electrons. $k=1$ for all methods except PM.}
    \label{fig:hub4410}
\end{figure}

From Figure~\ref{fig:hub4410}, we see that the convergence of the three
error measures share a similar pattern, thus the convergence of one
quantity should imply the convergence of the other two in practice. The
convergence of all methods tested consist of one or two stages: a
possible fast decay followed by linear convergence. We know that both
power method (PM) and gradient descent (Grad) converge linearly. In
this figure there are only $8$ iterates of PM since each iteration
needs the full access of the matrix. CD-Cyc-Grad and SCD-Uni-Grad seem
to behave similarly with slow linear decay, which is reasonable since
they treat each coordinate equally and should behave similarly as the
gradient descent method. CD-Cyc-LS and SCD-Uni-LS also decay linearly,
but with a faster rate. This is because updating coordinates by Grad
uses the same step size for all coordinates and all iterates, which
has to be small thus not optimal. LS always gives the fastest local
decay for different coordinates and different iterates. This argument
can be confirmed by the sudden-decay-behavior of CD-Cyc-LS. When
CD-Cyc-LS updates the important coordinate, it will give a large
`step size' and the error drops down rapidly.

All the other methods decay very fast at the beginning, and the common
feature is that they pick the coordinates for updating according to
their importance, instead of treating them equally. This justifies the
motivation of coordinate descent methods. Within $10^5$ column accesses,
equivalent to $\nicefrac{1}{12}$ iteration of PM, the projected energy
reaches $10^{-2}$ to $10^{-4}$ accuracy. This is really amazing,
and agrees with the result of the previous examples. GCD-Grad-LS and
GCD-LS-LS behave almost the same and they converge fastest both at the
initial stage and the linear convergence stage. The projected energy
reaches $10^{-8}$ accuracy when PM only iterates $3$ times. SCD-LS-LS
and SCD-Grad-LS also converge fast. The convergence rates of CPM and
SCPM are not as fast as SCD-LS-LS and SCG-Grad-LS, but they are still
much better than PM.

Comparing between the stochastic CDMs and the greedy CDMs, greedy
methods always converge a little faster, such as CD-Cyc-LS versus
SCD-Uni-LS and GCD-Grad-LS versus SCD-Grad-LS. This agrees with
our theoretical results and previous examples. If more than $1$
coordinate is updated at each iteration, greedy CDMs often get stuck
but stochastic CDMs are much less likely. Another thing to mention
is that in a small neighborhood of the minimum, updating by LS is
almost the same as updating by Grad with some step size depending on
the Hessian of the minimum. Thus, updating by Grad could perform as
well as LS ideally if there exists a step size which both guarantees
the convergence and also converges fast. On the other hand size,
LS choose a optimal step size every iteration.

\section{Conclusion}
\label{sec:conclusion}

We analyze the landscape of the non-convex objective function $f(x)$
of LEVP as \eqref{eq:LEVP-opt} and conclude that all local minima of
$f(x)$ are global minima.  We then investigate CD-Cyc-Grad and prove
the global convergence of CD-Cyc-Grad on $f(x)$ almost surely. Through
the derivation of the exact line search along a coordinate, GCD-LS-LS
is presented as the most greedy CDM and avoids many saddle points of
$f(x)$. Finally we propose SCD-Grad-vecLS($t$) and SCD-Grad-LS($t$). The
local convergence analysis for these stochastic CDMs shows that
convergence rate of either GCD-LS-LS or SCD-Grad-LS($t$) with $t>0$ is
provably faster or equal to that of SCD-Uni-LS. And one example as in
Figure~\ref{fig:gcdm-vs-scdm} (b), demonstrates the robustness of the
SCD-Grad-LS($t$) with small $t$. Therefore we recommend SCD-Grad-LS(1)
as an efficient and robust CDM for LEVP.  Numerical results agree with
our analysis. The performance of CDMs on a protypical quantum many-body
problem of the Hubbard model shows that CDMs has great potential for
quantum many-body calculations.

There are several directions worth exploring for future works.
Momentum acceleration for coordinate descent~\cite{Allen-Zhu2016,
Lee2013, Lin2015, Nesterov2012} can be combined with the proposed methods
to potentially accelerate the convergence at least in the local convex
area.  Also it is of interest to pursue the asynchronous implementation
of the proposed methods and prove the convergence property in that
setting. Another possible future direction is to employ the sparsity
of the matrix into the method such that per-iteration cost could be
further reduced.  Last but not least, the proposed methods should be
extensively tested on other quantum many-body calculations and more
other areas whose computational bottlenecks are the LEVP.


\medskip
\noindent
{\bf Acknowledgments.}  This work is partially
supported by the National Science Foundation under awards OAC-1450280
and DMS-1454939. We thank Stephen J. Wright, Weitao Yang and Wotao Yin
for helpful discussions.

\bibliographystyle{apalike} \bibliography{library}

\appendix

\section{Proof of Theorem~\ref{thm:CD-Cyc-Grad}} \label{app:CD-Cyc-Grad}

In order to show the global convergence of the CD-Cyc-Grad, we
introduce a few definitions and terminologies here. Let $\chi^s$
be a set of strict saddle points of $f(x)$, i.e., $\chi^s$ is
characterized by Lemma~\ref{lem:strict-saddle-point}. Let $g_j(x)$ be
the ``coordinate-updating'' in Algorithm~\ref{alg:CD-Cyc-Grad}, mapping
from $\vec{x}{\ell}{}$ to $\vec{x}{\ell+1}{}$ with coordinate $j$, i.e.,
\begin{equation} \label{eq:CD-Cyc-Grad-gj}
    \vec{x}{\ell+1}{} = g_{j}(\vec{x}{\ell}{}).
\end{equation}
When $n$ steps of updating are composed together, we denote the composed
mapping as
\begin{equation} \label{eq:CD-Cyc-Grad-g}
    g = g_{n} \circ g_{n-1} \circ \cdots \circ g_{1}.
\end{equation}
The corresponding iteration then is
\begin{equation}
    \vec{x}{(k+1)n}{} = g(\vec{x}{kn}{}), \qquad k = 0, 1, 2, \dots.
\end{equation}

We first establish a few lemmas for the proof of
Theorem~\ref{thm:CD-Cyc-Grad}.

\begin{lemma} \label{lem:CD-Cyc-Grad-LGW}
    Let $R \geq \sqrt{\max_j \norm{A_{:,j}}}$ be a constant, $\gamma
    \leq \frac{1}{4(n+4)R^2}$ be the stepsize, and $W_0 = \left\{ x :
    \norm{x}_\infty < R\right\}$ be the set of initial vectors. For any
    $x^{(0)} \in W_0$, and any coordinate index $j$, the following
    iteration is still in $W_0$, i.e.,
    \begin{equation}
        x^{(1)} = g_j(x^{(0)}) = x^{(0)} - \gamma \left( -4
        \vec{z}{0}{j} + 4 \norm{x^{(0)}}^2 \vec{x}{0}{j}\right)
        e_j \in W_0,
    \end{equation}
    where $g_j$ is defined as \eqref{eq:CD-Cyc-Grad-gj}.
\end{lemma}

\begin{proof}
    Since $x^{(0)}$ and $x^{(1)}$ only differ by one entry from each
    other, in order to validate that $x^{(1)} \in W_0$, it is sufficient
    to show that
    \begin{equation} \label{eq:CD-Cyc-Grad-bddL}
        \abs{ \vec{x}{0}{j} - \gamma \left( -4 \vec{z}{0}{j} + 4
        \norm{x^{(0)}}^2 \vec{x}{0}{j} \right) } < R.
    \end{equation}
    In the rest of the  proof, we drop the superscript $(0)$ to simplify
    the notations. We denote the expression in the absolute value in
    \eqref{eq:CD-Cyc-Grad-bddL} as
    \begin{equation}
        \func{h}{x_j} = -4\gamma x_j^3 + \left( 1 + 4\gamma A_{j,j} -
        4 \gamma \sum_{i \neq j}x_i^2 \right) x_j + 4 \gamma \sum_{i
        \neq j} A_{i,j} x_i,
    \end{equation}
    which is a cubic polynomial. Showing
    \eqref{eq:CD-Cyc-Grad-bddL} is equivalent to show that $-R <
    \func{h}{x} < R$ for any $x \in (-R,R)$. The local maximizer and
    minimizer of $h(x)$ are the roots of $h'(x)$,
    \begin{equation}
        x^\pm = \pm \sqrt{ \frac{1}{12\gamma} + \frac{A_{j,j}}{3} -
        \frac{\sum_{i \neq j} x_i^2}{3}},
    \end{equation}
    where we have used the fact that
    $1+4\gamma A_{j,j} - 4\gamma \sum_{i\neq j}x_i^2 \geq 0$, which
    can be verified using the constraints on $R$ and $\gamma$. Again,
    by the constraints, we obtain
    \begin{equation}
        \abs{x^\pm} \geq \sqrt{ R^2 + \frac{R^2+A_{j,j}}{3} +
        \frac{nR^2-\sum_{i \neq j}x_i^2}{3}} > R,
    \end{equation}
    which means the local maximizer and minimizer are beyond the
    interval $(-R,R)$. The leading coefficient of $h(x)$ is $-4\gamma
    < 0$.  Therefore, showing $-R < h(x) < R$ for any $x \in (-R,R)$
    can be implied by $h(R) < R$ and $h(-R) > -R$.
    
    The inequality $h(R) < R$ can be shown as
    \begin{equation}
        \begin{split}
            h(R) & =  R + 4 \gamma \left( \sum_i A_{i,j}x_i - R \sum_{i}
            x_i^2\right) \leq R + 4 \gamma \left( \norm{A_{:,j}}\norm{x}
            - R \norm{x}^2 \right) \\
          & \leq  R + 4 \gamma \norm{x} \left( \norm{A_{:,j}} - R^2
            \right) < R,
        \end{split}
    \end{equation}
    where the first inequality adopts H\"older's inequality and the
    second inequality is due to the fact that $\norm{x}\geq R$.  The
    inequality $h(-R) > -R$ can be shown analogously. Hence we proved
    the lemma.
\end{proof}

\begin{lemma} \label{lem:CD-Cyc-Grad-BddLip}
    For any $x\in W_0$,
    \begin{equation}
        \max_{i,j} \abs{e_i^\top \grad^2 f(x) e_j} < 4(n+3) R^2,
    \end{equation}
    where $W_0$ and $R$ are as defined in
    Lemma~\ref{lem:CD-Cyc-Grad-LGW}, $f$ is the objective function as
    in \eqref{eq:LEVP-opt}.
\end{lemma}

\begin{proof}
    For any $x \in W_0$ and any $i,j = 1, 2, \dots, n$, we have,
    \begin{equation}
        \abs{e_i^\top \grad^2 f(x) e_j} = \abs{ -4 A_{i,j} + 8 x_i x_j + 4
        \norm{x}^2 } < 4 \abs{A_{i,j}} + 8 R^2 + 4nR^2 < 4(n+3) R^2.
    \end{equation}
\end{proof}

One direct consequence of Lemma~\ref{lem:CD-Cyc-Grad-BddLip} is that
for any consecutive iterations in either CD-Cyc-Grad, we have the
following relation,
\begin{equation} \label{eq:Lip-ineq}
    \func{f}{ \vec{x}{\ell+1}{} } \leq \func{f}{\vec{x}{\ell}{}}
    +\grad_{j_\ell} \func{f}{\vec{x}{\ell}{}} \bigl(
    \vec{x}{\ell+1}{j_\ell} - \vec{x}{\ell}{j_\ell} \bigr) + \frac{L}{2}
    \abs{\vec{x}{\ell+1}{j_\ell} - \vec{x}{\ell}{j_\ell}}^2
\end{equation}
where $j_\ell$ is the chosen coordinate at the $\ell$-th iteration and $L
= 4(n+3)R^2$ is the Lipschitz constant for the gradient of $f$. To
simplify the presentation below, we will use $L$ instead of the explicit
expression.

\begin{lemma} \label{lem:CD-Cyc-Grad-grad0}
    Let $R \geq \sqrt{\max_j \norm{A_{:,j}}}$ be a constant, $\gamma
    \leq \frac{1}{4(n+4)R^2}$ be the stepsize, and $W_0 = \Set{ x
    | \norm{x}_\infty < R}$ be the set of initial vectors. For
    any $\vec{x}{0}{} \in W_0$, and the iteration follows either
    CD-Cyc-Grad, we have
    \begin{equation}
        \lim_{\ell \rightarrow \infty} \norm{ \grad
        \func{f}{\vec{x}{\ell}{}} } = 0.
    \end{equation}
\end{lemma}

\begin{proof}
    Substituting the updating expression in
    Algorithm~\ref{alg:CD-Cyc-Grad} into \eqref{eq:Lip-ineq}, we obtain,
    \begin{equation}
        \func{f}{ \vec{x}{\ell+1}{} } \leq \func{f}{\vec{x}{\ell}{}}
        - \gamma \bigl( \grad_{j_\ell} \func{f}{\vec{x}{\ell}{}}
        \bigr) ^2 + \frac{L\gamma^2}{2} \bigl( \grad_{j_\ell}
        \func{f}{\vec{x}{\ell}{}} \bigr)^2.
    \end{equation}
    Since $1 - \frac{\gamma L}{2} > 0$, we have for any $\ell$
    \begin{equation}
        \bigl( \grad_{j_\ell} \func{f}{\vec{x}{\ell}{}} \bigr)^2 \leq
        \frac{1}{\gamma \bigl( 1 - \frac{\gamma L}{2} \bigr)}
        \bigl( \func{f}{\vec{x}{\ell}{}} - \func{f}{\vec{x}{\ell+1}{}}
        \bigr).
    \end{equation}
    Summing over all $\ell$ from 0 to $T = nK$ for any big integer $K$,
    we have
    \begin{equation}
        \begin{split}
            \sum_{k=0}^{K} \sum_{\ell = nk }^{n(k+1)-1}\bigl(
            \grad_{j_\ell} \func{f}{\vec{x}{\ell}{}} \bigr)^2 \leq &
            \frac{1}{\gamma \bigl( 1 - \frac{\gamma L}{2} \bigr)}
            \bigl( \func{f}{\vec{x}{0}{}} - \func{f}{\vec{x}{T}{}}
            \bigr) \\
            \leq & \frac{1}{\gamma \bigl( 1 - \frac{\gamma L}{2} \bigr)}
            \bigl( \func{f}{\vec{x}{0}{}} - f^*
            \bigr).
        \end{split}
    \end{equation}
    where $f^*$ denotes the minimum of the objective function. Hence
    this means $\lim_{\ell \rightarrow \infty} \grad_{j_\ell} \func{f}{
    \vec{x}{\ell}{} } = 0$. The limit is equivalent to say that for
    any $\delta_0$, there exists a constant $K_0$ such that for any $k
    \geq K_0$, we have,
    \begin{equation} \label{eq:bound-gradf}
        \bigl\lvert \grad_{j_\ell} \func{f}{\vec{x}{\ell}{}}
        \bigr\rvert \leq \delta_0, \qquad \ell = kn,\dots,(k+1)n-1.
    \end{equation}
    Let $\ell_1$ and $\ell_2$ be two iterations within $kn$ and
    $(k+1)n-1$, we first conduct a loose Lipschitz bound on the
    coordinate-wise gradient, without loss of generality, we assume
    $\ell_1 \leq \ell_2$,
    \begin{equation}
        \begin{split}
            \abs{\grad_{j_{\ell_1}} \func{f}{ \vec{x}{\ell_1}{} } -
            \grad_{j_{\ell_1}} \func{f}{ \vec{x}{\ell_2}{} } } \leq &
            \sum_{\ell = \ell_1}^{\ell_2-1} \abs{\grad_{j_{\ell_1}}
            \func{f}{ \vec{x}{\ell}{} } - \grad_{j_{\ell_1}} \func{f}{
            \vec{x}{\ell + 1}{} } } \\
            \leq & L \sum_{\ell = \ell_1}^{\ell_2-1} \abs{
            \vec{x}{\ell}{} - \vec{x}{\ell + 1}{} } \\
            \leq & L \sum_{\ell = \ell_1}^{\ell_2-1} \abs{
            \gamma \grad_{j_\ell} \func{f}{ \vec{x}{\ell}{} } }
            \leq n \delta_0,
        \end{split}
    \end{equation}
    where the last inequality is due to \eqref{eq:bound-gradf}
    and $\gamma L \leq 1$.  Let $\ell_0$ be an iteration
    within $kn$ and $(k+1)n-1$. And another important point in
    Algorithm~\ref{alg:CD-Cyc-Grad} is that $\{j_\ell\}_{\ell =
    nk}^{(k+1)n-1}$ is $\{1,2,\dots,n\}$. Next, we would show
    that the square norm of the full gradient also converges to zero,
    \begin{equation}
        \begin{split}
            \norm{\grad \func{f}{\vec{x}{\ell_0}{}}}^2 = &
            \sum_{j = 1}^n \bigl( \grad_j \func{f}{\vec{x}{\ell_0}{}}
            \bigr)^2 \\
            = & \sum_{\ell = nk}^{(k+1)n-1} \bigl( \grad_{j_\ell}
            \func{f}{\vec{x}{\ell_0}{}} - \grad_{j_\ell}
            \func{f}{\vec{x}{\ell}{}} + \grad_{j_\ell}
            \func{f}{\vec{x}{\ell}{}} \bigr)^2 \\
            \leq & \sum_{\ell = nk}^{(k+1)n-1} \Bigl( \bigl(
            \grad_{j_\ell} \func{f}{\vec{x}{\ell}{}} \bigr)^2 + 2
            \delta_0 \abs{ \grad_{j_\ell} \func{f}{\vec{x}{\ell_0}{}}
            - \grad_{j_\ell} \func{f}{\vec{x}{\ell}{}} } \\
            & + \abs{ \grad_{j_\ell} \func{f}{\vec{x}{\ell_0}{}} -
            \grad_{j_\ell} \func{f}{\vec{x}{\ell}{}} }^2 \Bigr)\\
            \leq & ( n + 2n^2 + n^3) \delta_0^2.
        \end{split}
    \end{equation}
    Since $\delta_0$ can be arbitary small, therefore, we
    conclude that $\lim_{\ell \rightarrow \infty} \norm{\grad
    \func{f}{\vec{x}{\ell}{}} } = 0$.

    Here, the initial vector of the method must be chosen in the
    $W_0$ to guarantee the bounded Lipschitz condition on the gradient
    of $f$ and iteration stays within $W_0$.
\end{proof}

\begin{proof} [Proof of Theorem~\ref{thm:CD-Cyc-Grad}]
    
    Since $R \geq \sqrt{\max_j \norm{A_{:,j}}}$ and $\gamma
    \leq \frac{1}{4(n+4)R^2}$, Lemma~\ref{lem:CD-Cyc-Grad-LGW}
    states that for any $x \in W_0$ and $j$, we have $g_j(x) \in
    W_0$. Hence we have, for any $x \in W_0$, $g(x) \in W_0$, where
    $g$ is defined as \eqref{eq:CD-Cyc-Grad-g}. Further, as stated
    by Lemma~\ref{lem:CD-Cyc-Grad-BddLip}, $f$ has bounded Lipschitz
    coordinate gradient in $W_0$ and the stepsize $\gamma$ obeys $\gamma
    \leq \frac{1}{4(n+4)R^2} < \frac{1}{4(n+3)R^2}$. Proposition 4
    in \cite{Lee2017a} shows that under these conditions, $\det(D
    g(x)) \neq 0$. Corollary 1 in \cite{Lee2017a} shows that $\mu
    \bigl( \Set{\vec{x}{0}{} | \lim_{k \rightarrow \infty}
    g^k(\vec{x}{0}{}) \in \chi^s} \bigr) = 0$. Combining with the
    conclusion of Lemma~\ref{lem:CD-Cyc-Grad-grad0}, we obtain the
    conclusion of Theorem~\ref{thm:CD-Cyc-Grad}.

\end{proof}

\section{Proof of local convergence} \label{app:local-conv}

\begin{proof}[Proof of Lemma~\ref{lem:mono-decay}]
    For any $\vec{x}{\ell}{} \in D^+ \subset B^+$ and fixed index
    $j$, $\func{f}{\vec{x}{\ell}{}}$ is a quartic monic polynomial
    of $\vec{x}{\ell}{j}$, denoted as $p(\vec{x}{\ell}{j})$.  Since
    $\func{f}{\vec{x}{\ell}{}} = p(\vec{x}{\ell}{j})$, $\calI = \{ y:
    p(y) \leq \func{f}{\vec{x}{\ell}{}}\}$ is a non-empty set.
    
    When $p'(\vec{x}{\ell}{j}) = \grad_j \func{f}{\vec{x}{\ell}{}}
    = 0$, we obtain,
    \begin{equation}
        \func{f}{\vec{x}{\ell+1}{}} \leq \func{f}{\vec{x}{\ell}{}} -
        \frac{1}{2L} \left(\grad_j \func{f}{\vec{x}{\ell}{}}\right)^2.
    \end{equation}
    
    When $p'(\vec{x}{\ell}{j}) = \grad_j \func{f}{\vec{x}{\ell}{}} < 0$,
    there exists an interval $[a,b]$ with $a < b$ such that $p(a) = p(b)
    = \func{f}{\vec{x}{\ell}{}}$ and $p(y) < \func{f}{\vec{x}{\ell}{}}$
    for all $y \in (a,b)$. It can be further confirmed that $a =
    \vec{x}{\ell}{j}$. By mean value theorem, there exists at least
    one number $c \in (a,b)$ such that $p'(c) = 0$. If there are two,
    let $c$ denote the smaller one so that $p'(y) < 0$ for $y \in
    (a,c)$. Applying the mean value theorem one more time, we have,
    \begin{equation}
        0 < - \grad_j \func{f}{\vec{x}{\ell}{}} = - p'(a) = p'(c) -
        p'(a) = (c-a) p''(\xi) \leq (c-a)L,
    \end{equation}
    which means
    \begin{equation} \label{eq:inLipDomain}
        a - \frac{\grad_j \func{f}{\vec{x}{\ell}{}}}{L} =
        \vec{x}{\ell}{j} - \frac{\grad_j \func{f}{\vec{x}{\ell}{}}}{L}
        \in (a,c).
    \end{equation}
    Equation~\eqref{eq:inLipDomain} implies $\vec{x}{\ell}{} -
    \frac{1}{L}\grad_j \func{f}{\vec{x}{\ell}{}}e_j \in D^+$. Following
    the Lipschitz condition of $\grad_j f$ in $D^+$, we obtain,
    \begin{equation} \label{eq:fdecay}
        \begin{split}
            \func{f}{\vec{x}{\ell+1}{}} \leq & \func{f}{ \vec{x}{\ell}{}
            - \frac{1}{L}\grad_{j} \func{f}{ \vec{x}{\ell}{} } e_{j} } \\
            \leq & \func{f}{\vec{x}{\ell}{} } - \frac{1}{L} \left(
            \grad_{j} \func{f}{ \vec{x}{\ell}{} } \right)^2 +
            \frac{L}{2}\left( \frac{1}{L} \grad_{j} \func{f}{
            \vec{x}{\ell}{} } \right)^2 \\
            = & \func{f}{\vec{x}{\ell}{}} - \frac{1}{2L} \left( \grad_{j}
            \func{f}{\vec{x}{\ell}{}} \right)^2.
        \end{split}
    \end{equation}

    Similarly, equation~\eqref{eq:fdecay} can be obtain when
    $p'(\vec{x}{\ell}{j}) = \grad_j \func{f}{\vec{x}{\ell}{}} > 0$ as
    the case of $p'(\vec{x}{\ell}{j}) = \grad_j \func{f}{\vec{x}{\ell}{}}
    < 0$.
    
    Analogically, we can show the same conditions with the same constants
    hold for $f$ in $D^-$.

    Once we have $\func{f}{\vec{x}{\ell+1}{}} \leq
    \func{f}{\vec{x}{\ell}{}} - \frac{1}{2L}\bigl(
    \grad_j \func{f}{\vec{x}{\ell}{}}\bigr)^2$, it is
    straightforward to show that $\vec{x}{\ell+1}{} \in D^+
    \cup D^-$. If $\vec{x}{\ell+1}{} \not\in D^+ \cup D^-$ then
    there exist more than two minimizers of $f(x)$ which violates
    Theorem~\ref{thm:global-minimizer}.
    Hence, $\vec{x}{\ell+1}{} \in D^+ \cup D^-$.
\end{proof}

\begin{proof}[Proof of Lemma~\ref{lem:conv}]
    By Lemma~\ref{lem:mono-decay}, we have,
    \begin{equation} \label{eq:local-conv1}
        \func{f}{ \vec{x}{\ell+1}{} } \leq \func{f}{ \vec{x}{\ell}{} }
        - \frac{1}{2L} \left( \grad_{j_{\ell}} \func{f}{ \vec{x}{\ell}{}
        } \right)^2,
    \end{equation}
    where $j_\ell$ denotes the index picked at $\ell$-th iteration.
    Subtracting $f^*$ from both side of \eqref{eq:local-conv1} and
    taking the conditional expectation, we obtain,
    \begin{equation} \label{eq:local-conv2}
        \begin{split}
            \expect{ \func{f}{ \vec{x}{\ell+1}{} } \mid \vec{x}{\ell}{}
            } - f^* \leq & \func{f}{ \vec{x}{\ell}{} } - f^*
            - \frac{1}{2L} \expect{ \bigl( \grad_{j_{\ell}} \func{f}{
            \vec{x}{\ell}{} } \bigr)^2 \mid \vec{x}{\ell}{} } \\
            = & \func{f}{ \vec{x}{\ell}{} } - f^* -
            \frac{1}{2L} \sum_{j=1}^n \frac{\abs{\grad_j \func{f}{
            \vec{x}{\ell}{} } }^{t+2}}{ \norm{\grad \func{f}{
            \vec{x}{\ell}{} } }_t^t} \\
            = & \func{f}{ \vec{x}{\ell}{} } - f^* -
            \frac{1}{2L} \frac{ \norm{\grad \func{f}{ \vec{x}{\ell}{}
            } }_{t+2}^{t}}{ \norm{\grad \func{f}{ \vec{x}{\ell}{} }
            }_t^t} \norm{\grad \func{f}{ \vec{x}{\ell}{} } }_{t+2}^{2} \\
            \leq & \func{f}{ \vec{x}{\ell}{} } - f^* -
            \frac{1}{2L n^{\frac{2}{t+2}}} \norm{\grad \func{f}{
            \vec{x}{\ell}{} } }_{t+2}^{2},
        \end{split}
    \end{equation}
    where the first equality uses the probability $p_j$ of the sampling
    procedure and the last equality is due to elementary vector norm
    inequality.

    Next, we will bound the last term in \eqref{eq:local-conv2} with
    the \strongcvx{\frac{t+2}{t+1}} property of $f(x)$. Minimizing
    both side of \eqref{eq:strongly-convex} with respect to $y$
    and substituting $x = \vec{x}{\ell}{}$, we have,
    \begin{equation} \label{eq:local-conv3}
        \begin{split}
            f^* \geq & \func{f}{ \vec{x}{\ell}{} } -
            \sup_y \left\{ -\grad \func{f}{ \vec{x}{\ell}{} }^\top
            \left( y - \vec{x}{\ell}{} \right) - \frac{\mu_q}{2}
            \norm{y-\vec{x}{\ell}{}}_{\frac{t+2}{t+1}}^2 \right\} \\
            \geq & \func{f}{ \vec{x}{\ell}{} } - \frac{1}{2\mu_q} \norm{
            \grad \func{f}{ \vec{x}{\ell}{} } }_{t+2}^2, \\
        \end{split}
    \end{equation}
    where $q = \frac{t+2}{t+1}$ and the last supreme
    term in the first line is a conjugate of function
    $\frac{\mu_q}{2}\norm{\cdot}_{\frac{t+2}{t+1}}^2$ and the
    result is given by Example 3.27 in~\cite{Boyd2004}, which is
    $\frac{1}{2\mu_q}\norm{\cdot}_{t+2}^2$. Here $\norm{\cdot}_{t+2}$
    is the dual norm of $\norm{\cdot}_{\frac{t+2}{t+1}}$.

    Substituting \eqref{eq:local-conv3} into \eqref{eq:local-conv2}, we
    obtain,
    \begin{equation}
        \expect{ \func{f}{ \vec{x}{\ell+1}{} } \mid \vec{x}{\ell}{} }
        - f^* \leq \bigl( 1 - \frac{\mu_q}{L n^{2 - 2\frac{t+1}{t+2}}}
        \bigr) \bigl( \func{f}{ \vec{x}{\ell}{} } - f^* \bigr)
         = \bigl( 1 - \frac{\mu_q}{L n^{2-\frac{2}{q}}} \bigr)
        \bigl( \func{f}{ \vec{x}{\ell}{} } - f^* \bigr).
    \end{equation}
\end{proof}

\begin{proof}[Proof of Theorem~\ref{thm:scd-conv}]
    By Lemma~\ref{lem:mono-decay}, we know that if $\vec{x}{0}{}
    \in D^+ \cup D^-$, then $\vec{x}{\ell}{} \in D^+ \cup
    D^-$ for all $\ell$. Hence Lemma~\ref{lem:conv} holds for all
    $\vec{x}{\ell}{} \in D^+ \cup D^-$. We take expectation
    condition on the sigma algebra of $\vec{x}{0}{}$,
    \begin{equation}
        \begin{split}
            \expect{ \func{f}{\vec{x}{\ell}{}} \mid \vec{x}{0}{} } -
            f^* \leq & \bigl( 1 - \frac{\mu_q}{L n^{2-\frac{2}{q}}}
            \bigr) \left( \expect{ \func{f}{\vec{x}{\ell-1}{}} \mid
            \vec{x}{0}{} } - f^* \right) \\
            \leq & \cdots \leq \bigl( 1 - \frac{\mu_q}{L
            n^{2-\frac{2}{q}}} \bigr)^\ell \left(
            \func{f}{\vec{x}{0}{}} - f^* \right),
        \end{split}
    \end{equation}
    where we adopt property of conditional expectation.

    Assume $\vec{x}{\ell}{} \in D^+$ and $x^* = \sqrt{\lambda_1}v_1 \in
    D^+$. Due to the strongly convexity of $f(x)$, we have
    \begin{equation}
        \begin{split}
            \expect{ \dist{ \vec{x}{\ell}{} }{ X^* }^2 \mid
            \vec{x}{0}{} } \leq & \expect{ \norm{ \vec{x}{\ell}{}
            - x^* }^2 \mid \vec{x}{0}{} } \\
            \leq & \frac{2}{\mu_2} \expect{  \func{f}{\vec{x}{\ell}{}}
            - f(x^*) \mid \vec{x}{0}{} } \\
            \leq & \frac{2}{\mu_2} \bigl(
            1-\frac{\mu_q}{Ln^{2-\frac{2}{q}}} \bigr)^\ell \left(
            \func{f}{\vec{x}{0}{}} - f^* \right).
        \end{split}
    \end{equation}

    If $\vec{x}{\ell}{} \in D^-$, then we choose $x^* = -\sqrt{\lambda_1}
    v_1 \in D^-$ and the conclusion follows as above.
\end{proof}

\end{document}